\documentclass[11pt, reqno]{amsart}
\usepackage{amsmath}
\usepackage{amsthm}
\usepackage{amssymb}
\usepackage{amscd}
\usepackage{amsbsy}
\usepackage{epsfig}
\usepackage{graphicx}
\usepackage{psfrag}

\usepackage{bbm}
\usepackage{hyperref}
\usepackage{geometry}
\newtheorem{theorem}{Theorem}[section]
\newtheorem{remark}[theorem]{Remark}
\newtheorem{proposition}[theorem]{Proposition}
\newtheorem{lemma}[theorem]{Lemma}

\def\H#1{{\bf #1}}

\def\ie{\emph{i.e.}, }

\newfont\bbf{msbm10 at 12pt}
\def\eps{\varepsilon}
\def\R{{\mathbb R}}

\def\N{{\mathbb N}}
\def\Z{{\mathcal Z}}

\def\B{{\mathcal B}}
\def\P{{\mathcal P}}

\def\G{{\mathcal G}}
\def\D{{\mathcal D}}
\def\H{{\mathcal H}}

\def\L{{\mathcal L}}
\def\Q{{\mathcal Q}}
\def\cont{{\mathcal C}}
\def\es{{\emptyset}}
\def\sm{\setminus}

\def\dist{\mbox{dist}}

\def\diam{\mbox{\rm diam} }

\def\bd{\partial }
\def\le{\leqslant}
\def\ge{\geqslant}

\newcommand{\Crit}{\mbox{Crit}}
\newcommand{\sCrit}{\mbox{\scriptsize Crit}}
\newcommand{\I}{\mathring{I}}
\newcommand{\f}{\mathring{f}}
\newcommand{\hDelta}{\mathring{\Delta}}
\newcommand{\tH}{\tilde{H}}

\newcommand{\Lp}{\mathcal{L}}
\newcommand{\bm}{\overline{m}}
\newcommand{\bmu}{\overline{\mu}}
\newcommand{\vf}{\varphi}
\newcommand{\ve}{\varepsilon}

\newcommand{\tmu}{\tilde{\mu}}
\newcommand{\tg}{\tilde{g}}
\newcommand{\tnu}{\tilde{\nu}}
\newcommand{\teta}{\tilde{\eta}}
\newcommand{\hLp}{\mathring{\Lp}}
\newcommand{\hF}{\mathring{F}}
\newcommand{\lip}{\mbox{\tiny Lip}}

\newcommand{\dlj}{\Delta_{\ell,j}}
\newcommand{\pa}{\mathcal{P}}

\newcommand{\e}{\mathfrak{e}}
\newcommand{\beq}{\begin{equation}}
\newcommand{\eeq}{\end{equation}}

\def\M{\mathcal{M}}
\def\cyl{{\rm C}}

\def\nolift{\mathcal{NR}}

\begin{document}

\title[Equilibrium states for multimodal maps with holes]{Equilibrium states, pressure and escape for multimodal maps with holes}
\author[M. Demers]{Mark Demers}
\address{Mark Demers\\ Department of Mathematics and Computer Science\\
Fairfield University\\
Fairfield, CT 06824 \\
USA}\email{\href{mailto:mdemers@fairfield.edu}{mdemers@fairfield.edu}}
\urladdr{\url{http://www.faculty.fairfield.edu/mdemers/}}

\author[M. Todd]{Mike Todd}
\address{Mike Todd\\ Mathematical Institute\\
University of St Andrews\\
North Haugh\\
St Andrews\\
KY16 9SS\\
Scotland} \email{\href{mailto:m.todd@st-andrews.ac.uk}{m.todd@st-andrews.ac.uk}}
\urladdr{\url{http://www.mcs.st-and.ac.uk/~miket/}}

\begin{abstract}
For a class of non-uniformly hyperbolic interval maps,
we study rates of escape with respect to conformal measures
associated with a family of geometric potentials.  We establish the existence of physically relevant
conditionally invariant measures and equilibrium states and prove a relation
between the rate of escape and pressure with respect to these potentials.    
As a consequence, we obtain a Bowen formula: we express the Hausdorff dimension 
of the set of points which never exit through the hole in terms of the relevant 
pressure function.  Finally, we obtain an expression for the derivative
of the escape rate in the zero-hole limit.
\end{abstract}

\thanks{MD was partially supported by NSF grant DMS 1101572.  MT was partially supported by NSF grants DMS 0606343 and DMS 0908093.}

\date{\today}
\maketitle

\section{Introduction}
\label{sec:intro}

For a class of dynamical systems with holes, we study the relation between the
conditionally invariant measures, rates of escape and pressures with respect to  
a family of potentials.  Given an interval map $f : I \circlearrowleft$ and a hole $H \subset I$,
we define the exponential rate of escape with respect to a reference measure $m$ to be
\beq
\label{eq:escape}
\e(m,H) = - \lim_{n \to \infty} \frac 1n \log m(\cap_{i=0}^n f^{-i}(I\setminus H))
\eeq
when the limit exists.  
We say the open system satisfies a   {\em Variational Principle}
with respect to a potential $\phi$
if $-\e(m,H) = P_\cont(\phi)$ where $P_\cont(\phi)$ denotes the pressure
of $\phi$ taken over a class of relevant invariant measures $\cont$, 
\[
P_\cont(\phi) = \sup_{\mu \in \cont} \left\{ h(\mu) + \int \phi \, d\mu \right\}.
\]

We will focus on a class of multimodal Collet-Eckmann maps of the interval
satisfying a slow-recurrence condition to the boundary of the hole.  Such maps were
studied in \cite{BDM} using Lebesgue measure as a reference measure and 
$- \log |Df|$ as the relevant potential.

In  this paper, we generalize this study to include the family of potentials
$\{\varphi_t:=-t\log|Df|:t\in \R\}$.  
We will denote by $P_{\M_f}(\varphi_t)$
the pressure with respect to the potential $\varphi_t$ taken over all ergodic $f$-invariant
probability measures, $\M_f$.  
These potentials are often referred to as \emph{geometric potentials} since they capture the geometry and statistical growth properties of the system.   For example, it was shown in \cite{Led81} that a measure $\mu\in \M_f$ with positive entropy is an equilibrium state for 
$\varphi_1$ if and only if $\mu$ is 
absolutely continuous with respect to Lebesgue measure.  Moreover, it was shown in \cite{BrKell} (unimodal Collet-Eckmann case, restricted $t$), \cite{BTeqnat} (multimodal case, restricted $t$) and \cite{ITeq} (multimodal case, general $t$) that there is 
an equilibrium state $\mu_t$ corresponding to $\varphi_t$.   The relation between these measures, the pressure and the Lyapunov spectrum was shown in \cite{ITdim}. 
The classical Bowen formula in the uniformly expanding case, see for example \cite{Rai89, LivMau03}, states that the Hausdorff dimension of the survivor set (the set of points which never escapes through the hole)  is the value $t^*\ge 0$ such that $P_{\M_f}(\varphi_{t^*})=0$.

In this paper we will fix a relevant reference measure and then look at how the mass given by this measure escapes through holes.   
When we consider the potentials $\varphi_t$ for $t\in \R$, our reference measure will be the corresponding $(\varphi_t- P_{\M_f}(\varphi_t))$-conformal measure $m_t$.  These were shown to exist in \cite{ITrev}, and moreover for the equilibrium state $\mu_t$ for $\varphi_t$, we have $\mu_t\ll m_t$. 

For this class of potentials and reference measures, we prove that the escape rate has a natural expression in terms of the pressure; we also prove the existence of further measures, one of which is an equilibrium state on the survivor set
and one of which is the relevant  `geometric conditionally invariant measure' for the system.  
Such conditionally invariant measures, defined precisely in Section~\ref{ssec:transfer}, 
describe the evolution of reasonable 
classes of initial
distributions that have densities with respect to the conformal measures $m_t$.
In addition, we are able to prove a Bowen formula for the Hausdorff dimension of the survivor set.
 Finally, we provide a formula for the derivative of the
escape rate as our hole shrinks to a point (the zero-hole limit).

Similar results regarding the derivative of the escape rate were proved in \cite{KellLiv, FergPoll}
using spectral theory.  By contrast, in our setting no spectral picture is known for the transfer
operators associated with our class of multimodal maps, so we construct Young towers instead.
Unfortunately, a new Young tower must be constructed for each hole since return times can 
suffer unbounded changes due to arbitrarily small perturbations.  Thus a principal aim of the 
present paper is to develop techniques which allow us to retain sufficient control of the
towers we construct along a sequence of holes to prove results such as the Bowen formula
and the derivative of the escape rate.  We note that questions in thermodynamic formalism, such as multifractal spectra have been studied before in the context of multimodal maps, for example in \cite{ChuTak14, ITdim}, this is the first proof of a Bowen formula for such a general class of maps.  Moreover, our development of the theory of Young towers to tackle this problem gives a powerful abstract framework to deal with other non-uniformly hyperbolic dynamical systems. 

The paper is organized as follows.  In Section~\ref{sec:setup}, we precisely define our class of
maps, introduce our conditions on the types of holes we allow and recall definitions of
the objects fundamental to the present work, including pressure, inducing schemes and Young towers.
Section~\ref{sec:results} contains a precise statement of our main results while Section~\ref{sec:tails}
establishes that we have uniform control over our inducing schemes for
a family of potentials.  In Section~\ref{sec:LY} we recall some facts from \cite{BDM} regarding
abstract towers with holes and in Section~\ref{sec:project} we show how to apply those results to our
present setting.  Finally, we prove our Variational Principle in Section~\ref{sec:var proof}
and a Bowen formula for the Hausdorff dimension of the survivor set in Section~\ref{sec:bowen}.
  Section~\ref{zero hole} contains the proof of the derivative of the escape rate
in the zero-hole limit.


\section{Setup}
\label{sec:setup}

\subsection{Multimodal interval maps with some exponential growth} 

Collet-Eckmann maps are interval maps $f: I \circlearrowleft$ with critical points such that
the derivatives $Df^n$ at the critical values increase exponentially.  We will follow the approach of \cite{BDM, DHL} which allows for multimodal
maps with singularities (in this case our singularity set will be the boundary of the hole).

We say a critical point $c$ has {\em critical order} $\ell_c>0$ if there exists a neighborhood
$U_c$ of $c$ and a diffeomorphism $g_c: U_c \to g_c(U_c)$ such that $g_c(c)=0$ and
$f(x) = f(c) \pm |g_c(x)|^{\ell_c}$ for all $x \in U_c$.
A critical point $c$ is \emph{non-flat} if $\ell_c<\infty$. 
 
In this paper, we assume the map $f:I \to I$ is topologically mixing and $\cont^2$ with a critical set
$\Crit_c$ consisting of finitely many critical points $c$ with critical order $2 \le \ell_c < \infty$. Note that in particular, topological mixing means our maps are non-renormalizable: we make this assumption to avoid technicalities regarding uniqueness of equilibrium states.  
Let $B_\delta(\Crit_c) = \cup_{c \in \sCrit_c} B_\delta(c)$ denote the 
$\delta$-neighborhood of $\Crit_c$.  We assume $f$ satisfies the following
conditions for all sufficiently small $\delta > 0$:

\begin{enumerate}
  \item[{\bf(C1)}]  {\em Expansion outside $B_\delta(\Crit_c)$: }  There exist
$\gamma > 0$ and $\kappa > 0$ such that for every $x$ and $n \geq 1$ such that
$x_0 = x, \dots, x_{n-1} = f^{n-1}(x) \notin B_\delta(\Crit_c)$, we have
\[
|Df^n(x)| \geq \kappa \delta^{\ell_{\max}-1} e^{\gamma n},
\]
where $\ell_{\max} = \max\{ \ell_c : c \in \Crit_c \}$.
Moreover, if  $x_0 \in f(B_\delta(\Crit_c))$ or $x_n \in B_\delta(\Crit_c)$,
then we have
\[
|Df^n(x)| \geq \kappa e^{\gamma n}.
\]
\item[{\bf(C2)}]  {\em Slow recurrence and derivative growth along critical orbit: } There exists $\Lambda > 0$ such that for all $c \in \Crit_c$
there is $\vartheta_c \in (0,\Lambda/(5\ell_c))$ such that
\[
|Df^k(f(c))| \geq e^{\Lambda k}
\ \text{ and }
\mbox{\textnormal{dist}}(f^k(c), \Crit_c) > \delta e^{-\vartheta_c k}
\quad \text{ for all } k \geq 1.
\]
\end{enumerate}

A consequence of (C1) and (C2) together is that all periodic orbits must be repelling.  The first half of condition (C2) is the actual Collet-Eckmann condition, and the second half is a slow
recurrence condition.\footnote{\cite{DHL} and \cite{BDM} include a third condition
as part of their formal assumptions: (C3) $\exists \, c^* \in \Crit_c$ whose preimages are dense in
$I$ and no other critical point is among these preimages.  In our setting, it follows from
(C1)-(C2) and our assumption of topological mixing that all $c \in \Crit_c$ satisfy this
condition.}   

We assume without loss of generality that $\vartheta_c$ is small relative to $\gamma$
and $\Lambda$. 

\subsection{Introduction of Holes}
\label{ssec:intro holes}

A hole $H$ in $I$ is a finite union of open intervals $H_j$,
$j = 1, \ldots, L$.
Let $\I = I \backslash H$ and set $\I^n = \bigcap_{i=0}^n f^{-i}\I$, $n \in \N \cup \{ \infty \}$.
We refer to the set $\I^\infty$ as the {\em survivor set} for the open system, \ie $\I^\infty$
represents the set of points that do not escape in forward time.
Define $\f^n = f^n|\I^n$, $n \ge 1$, to be the maps on the noninvariant domains $\I^n$.

Two objects fundamental to the study of open systems are the escape rate $\e(m,H)$
defined by \eqref{eq:escape} and conditionally
invariant measures, whose definition we recall below.
In what follows, in order to simplify notation
when the hole  is clear by context, 
we sometimes suppress that variable and denote the escape rate by $\e(m)$.  

A {\em conditionally invariant measure} for the open system $(I,f,H)$ 
 is a probability measure $\mu$ for which there exists a constant $0 \le \lambda < 1$ such that
$\f_*\mu(A) := \mu(f^{-1}A \cap \I^1) = \lambda \mu(A)$ for any Borel set $A \subset I$.
This relation immediately implies $\lambda = \mu(\I^1)$ and $\f^n_*\mu(A) = \lambda^n \mu(A)$
so that $\e(\mu) = -\log \lambda$ by \eqref{eq:escape}.  

In order to invoke the tower construction of \cite{BDM},
we place several conditions
on the placement of the holes in the interval $I$.

\medskip
\noindent
\parbox{.1 \textwidth}{\bf(H1)}
\parbox[t]{.8 \textwidth}{Let $\vartheta_c, \delta >0$ be as in (C2).
For all $c \in \Crit_c$ and $k \geq 0$,
\[
\mbox{\textnormal{dist}}(f^k(c), \partial H) > \delta e^{-\vartheta_c k}.
\]
}

\medskip
\noindent
(H1) and (C2) imply that we can treat $\partial H$ the same as $\Crit_c$ in terms of the
slow approach of critical orbits.
Our second condition on $H$ is that the positions of its connected
components are generic with respect to
one another.  This condition also doubles as a transitivity condition
on the constructed tower
which ensures our conditionally invariant densities will be bounded away from
zero.  In order to
formulate this condition, we need the following fact about
$\cont^2$ nonflat nonrenormalizable maps satisfying (C1)-(C2)
(see \cite[Sect. 2.4.2]{BDM} or \cite[Lemma 1]{DHL}).
\begin{equation}
\label{eq:no holes mixing}
\begin{array}{l}
\text{There exist $c^*\in \Crit$ and $\delta_* >0$ such that for all $\delta \le \delta_*$, there exists 
$n = n(\delta)$}\\ [0.1 cm]
\text{such that for all
intervals $\omega \subseteq I$ with }
|\omega| \geq \frac{\delta}{3}, \\[0.1cm]
\quad\ (i)\ f^n\omega \supseteq I, \text{ and }  \\[0.1cm]
\quad (ii)\ \text{there is a subinterval $\omega' \subset \omega$ such that
$f^{n'}$ maps $\omega'$ }\\[0.01cm]
\quad \quad\ \text{ diffeomorphically onto }(c^*-3\delta, c^*+3\delta) \text{ for some } 0 < n' \leq n.
\end{array}
\end{equation}
Using this fact, we formulate a condition on the placement of the components
of the hole.  This condition is generic in the sense that it is satisfied by a full-measure set
of parameters governing the placement.
Within each component $H_j$, we place an artificial critical point
$b_j$, so
\[
\Crit_{\mbox{\tiny hole}} := \{ b_1, \dots, b_L \}.
\]
The points $b_j$ are chosen so that  $\Crit_{\mbox{\tiny hole}} \cap \Crit_c = \emptyset$.  
Choose $\delta$ so small that all points in $\Crit_c \cup \Crit_{\mbox{\tiny hole}}$
are at least $\delta$ apart and let $n(\delta)$ be the corresponding integer from
\eqref{eq:no holes mixing}.  We assume the following.

\medskip

\noindent
\parbox{.07 \textwidth}{\bf(H2)}
\parbox[t]{.9 \textwidth}{(a) $(\cup_{n\ge 0} f^nb_j) \cap c \in \Crit_c = \emptyset$ for all
$1 \leq j \leq L$. 
\\ [0.1cm]
(b) Let $f^{-1} (f b_j) = \cup_{i=1}^{K_j} g_{j,i}$.
For all $j,k \in \{1, \ldots, L\}$, there exists $i \in \{1, \ldots, K_j \}$
such that
$f^\ell b_k \neq g_{j,i}$ for
$1 \leq \ell \leq n(\delta)$. \\ [0.1cm]
(c) For each $j = 1, \ldots, L$, there is $r= r(j)$ such that
for all $x \in B_\delta(b_j)$, 
$f^i(x) \notin B_\delta(\Crit_c \cup \Crit_{\mbox{\tiny hole}})$ for $i = 1, \ldots, r(j)-1$,
and $|Df^r(x)| \geq \max\{ \kappa e^{\gamma r}, 4 \}$.}
\medskip

\noindent
For generically placed holes,
Condition (C1) implies
$|Df^r (x)| \geq \kappa e^{\gamma r}$ whenever
$x \notin B_\delta(\Crit_c)$ and
$f^r(x) \in B_\delta(\Crit_c)$, so by taking $\delta$ small,
and using assumption (H2)(a), we can always satisfy (H2)(c).
The specific form of (H2)(c) is to allow the $b_j$ to be periodic points,  
which is the one point of difference with \cite{BDM} in this condition.


\subsection{Pressure and conformal measures}
\label{ssec:intro press}

Suppose that $f:X\to X$ is a dynamical system on a topological space $X$ and $\phi:X\to [-\infty, \infty]$ is a \emph{potential}, both of these maps preserving the Borel structure.   Then we define the \emph{pressure} of $\phi$ to be 

$$P_{\M_f}(\phi):=\sup_{\mu\in \M_f}\left\{h(\mu)+\int\phi~d\mu:-\int\phi~d\mu<\infty\right\},$$
where
$$\M_f:=\left\{\mu \text{ Borel, ergodic}, \ \mu\circ f^{-1}=\mu, \ \mu(X)=1\right\},$$
and $h(\mu)$ denotes the (metric) entropy of $\mu$.  If a measure $\mu\in \M_f$ satisfies $h(\mu)+\int\phi~d\mu=P_{\M_f}(\phi)$, then we call $\mu$ an \emph{equilibrium state} for $(X,f,\phi)$.  These measures are often associated with another natural type of (possibly non-invariant) measure:  a Borel measure $m$ on $X$ is called \emph{$\phi$-conformal} if the Jacobian of $m$ is $e^\phi$, \ie 
$\frac{dm}{d(m\circ f)}=e^\phi$.

In this paper we will be particularly interested in the set of interval maps defined above and in the potential $\varphi:=-\log|Df|$ as well as the family  
$$\{\varphi_t:=-t\log|Df|:t\in \R\}.$$ 
We will sometimes denote $p_t:=P_{\M_f}(\varphi_t)$
to be the pressure with respect to the potential $\varphi_t$.  
 
For the potential $\varphi = \varphi_1$, the natural reference measure is 
$m=$ Lebesgue, with respect to which the equilibrium state $\mu$ for $(I, f, \varphi)$ is absolutely continuous.  Notice that Lebesgue is $\varphi$-conformal, and indeed since 
$p_1=0$, it is trivially $(\vf-p_1)$-conformal.  This case was studied in \cite{BDM}.  
When we consider the potentials $\varphi_t$ for $t\in \R$, our reference measure will be the corresponding $(\varphi_t- p_t)$-conformal measure $m_t$.  These were shown to exist in \cite{ ITrev}, and moreover the equilibrium state $\mu_t$ for $\varphi_t$
satisfies $\mu_t\ll m_t$. 
For convenience, for $t\in \R$ we denote
$$\phi_t:=\vf_t-p_t.$$

Given a potential $\phi$ on $I$, when we introduce a hole $H$ into the interval, we 
define the corresponding punctured potential by
$\phi^H(x)= \phi(x)$ on $I\sm H$ and $\phi^H(x)=-\infty$ on $H$;
the corresponding set of measures is
$$\M_f^H:=\left\{\mu\in\M_f:\mu(H)=0\right\}.$$
Observe that by invariance, these measures must be supported on $\I^\infty$.

In order for the pressure with respect to our punctured potential to be well-defined, we will
have to restrict our class of invariant measures further.  Define 
\beq
\label{eq:class G}
\G_f^H = \{ \mu \in \M^H_f : \exists C, \beta >0 \mbox{ such that for all } \ve >0, 
\mu(B_\ve(\partial H)) \le C \ve^\beta \} .
\eeq
The corresponding pressure we shall work with is,
\[
P_{\G^H_f}(\phi) = \sup_{\mu \in \G^H_f} \left\{ h(\mu) + \int \phi \, d\mu : 
-\int \phi \, d\mu < \infty \right\}.
\]
The class of measures $\G^H_f$ are those invariant measures which do not concentrate
too much mass on the boundary of the hole.
As we will see in Section~\ref{sec:var proof}, the relevant measures here, for example the equilibrium state for $\phi_t^H$, lie in $\G^H_f$, so focusing on these measures is not a significant restriction.


\subsection{Transfer Operators}
\label{ssec:transfer}

We will study the statistical properties of our open systems via  transfer
operators both for $(I,f,H)$ and for the associated Young tower, defined in the next section.

Given a potential $\phi$ on $I$ and a suitable test function $\psi$, the associated 
transfer operator $\Lp_\phi$ acts on $\psi$ by $\Lp_\phi \psi(x) = \sum_{y \in f^{-1}x} \psi(y) e^{\phi(y)}$.
When we work with the corresponding punctured potential $\phi^H$, we define
the transfer operator in terms of the restricted map $\f$:
\[
\Lp_{\phi^H} \psi(x)= \sum_{y \in \f^{-1}x} \psi(y) e^{\phi^H(y)} = \sum_{y \in f^{-1}x}
\psi(y) e^{\phi(y)} 1_{\I^1}(y),
\]
where $1_A$ denotes the indicator function of the set $A$.

The importance of the transfer operator stems from the fact that if $m$ is $\phi$-conformal
and $g$ is a function such that $\Lp_{\phi^H} g = \lambda g$ for some $\lambda >0$, then
$g m$ defines a conditionally invariant measure for $(I,f,H)$ with eigenvalue $\lambda$:
\[
gm(\f^{-1} A) = \int_{f^{-1} A \cap \I^1} g \, dm = \int_{A \cap \I} \Lp_{\phi^H} g \, dm 
= \lambda \int_{A \cap \I} g \, dm = \lambda \, gm(A),
\]
where in the last step we have used the fact that $g$ is necessarily zero on $H$ due to the
relation $\Lp_{\phi^H} g = \lambda g$.

Since for given $(I,f,H)$,
many conditionally invariant measures  exist for any eigenvalue between 0 and 1
under very mild conditions \cite{DemY}, it is imperative to find a
conditionally invariant measure with physical properties, such as 
that $\mu$ is the limit of $\f_*^nm/|\f_*^nm|$ for $m$ in
a reasonable class of initial distributions.

When our reference measure is the conformal measure $m_t$, we will take as our
class of initial distributions those measures $\eta$ having H\"older continuous densities with
respect to $m_t$.  If a conditionally invariant measure $\mu_t^H$
can be realized as the limit of $\f_*^n \eta/ |\f_*^n \eta|$ for all such measures $\eta$, then
we will call $\mu_t^H$ a {\em geometric conditionally invariant measure}.  
When $t=1$, such measures have been termed `physical conditionally invariant measures'
(see \cite{BDM, DemY});
we prefer the term `geometric' in this context, since although we will prove that
such $\mu_t^H$ have densities with respect to $m_t$, they are singular with respect to
Lebesgue measure when $t \neq 1$.


\subsection{Induced maps and Young towers}
\label{ssec:induced}

Given a set $\Crit_{\mbox{\tiny hole}}$ satisfying assumption (H2), in \cite{BDM}, 
inducing schemes $(X,F, \tau, H)$ are constructed respecting small holes $H$ 
satisfying (H1) and (H2).   
For an interval $X \subset I$, the triple $(X,F,\tau)$ is an \emph{inducing scheme} if there is  
a countable collection of subintervals $\{ X_i\}_i\subset X$ and a function $\tau:\cup_iX_i\to \N$ 
such that for each $i$, $\tau|_{X_i}$ is constant and the map $F=f^\tau|_{X_i}$  is a diffeomorphism
of $X_i$ onto $X$.  We define  $\tau_i:=\tau|_{X_i}$.

By `respecting the hole' $H$, we mean that for
each domain $X_i$ in the inducing scheme, either $f^n(X_i) \subset H$ or
$f^n(X_i) \cap H = \emptyset$ for $0 \le n \le \tau(X_i)$. 
To accomplish this, $\partial H$ is considered
as a discontinuity set for $f$ and cuts are introduced during the construction of the
inducing scheme.  
During the construction, no escape is allowed and the holes are inserted afterwards
into the tower $\Delta$ defined below.  

Given an inducing scheme respecting a hole $H$, $(X, F, \tau, H)$, we define the
corresponding Young tower as follows.  Let 
\[
\Delta = \{ (x, n) \in X \times \mathbb{N} \mid n < \tau(x) \} .
\]
$\Delta$ is viewed schematically as a tower with $\Delta_\ell = \Delta|_{n = \ell}$ as the
$\ell$th level of the tower.  The tower map, $f_\Delta$, is defined by $f_\Delta(x, \ell) = (x, \ell + 1)$
if $\ell + 1 < \tau(x)$ and $f_\Delta(x, \tau(x) -1) = (f^\tau(x), 0) = (F(x), 0)$ at return times.
There is a canonical projection $\pi : \Delta \to I$ satisfying $\pi \circ f_\Delta = f \circ \pi$.
$\Delta_0$ is identified  with $X$ so that $\pi|_{\Delta_0} = Id$.  The partition $\{ X_i \}$
induces a countable Markov partition $\{ \Delta_{\ell,j} \}$ on $\Delta$ via the identification
$\Delta_{\ell,j} = f_{\Delta}^\ell(X_j)$, for $\ell < \tau(X_j)$.  
The towers constructed in \cite{BDM} are mixing, \ie g.c.d.$\{\tau\} =1$,
and the partition $\{ X_i \}$ is generating. 

If $f^n(x) \in H$, then we place a hole $\tH$ in $\Delta_n$ and the elements above $\tH$
in the tower are deleted: \ie the set that
maps into $\tH$ does not return to the base.  The fact that the inducing
scheme respects $H$ implies that $\tH := \pi^{-1} H$ is the union of countably many
partition elements $\Delta_{\ell,j}$.  We set $\hDelta = \Delta \setminus \tH$ and refer to
the corresponding partition elements as $\hDelta_{\ell,j}$.
Similarly, we define $\hDelta^n = \cap_{i=0}^n f_\Delta^{-i} \hDelta$, $n \in \N \cup \{ \infty \}$.

\subsection{Lifting to the Young tower}
\label{ssec:lifting to Delta}

Since our main results are all proved using a Young tower $\Delta$ coming from an inducing scheme $(X, F, \tau, H)$, we will need to ensure that the tower we choose `sees all the relevant statistical properties' of our system.  
In particular, we will show that the Hausdorff dimension of the set of points which do not return to $X$ with $F$:
$$\nolift_\Delta := \left\{x\in X:\tau(x)=\infty\right\}$$
is a set with Hausdorff dimension `sufficiently bounded away from 1'; indeed Theorem~\ref{thm:lift} below bounds this by some $D<1$. 
This means that the Young tower contains all information in $I$ of sufficiently high Hausdorff 
dimension: we call points $x\in \pi(\Delta)$ \emph{liftable} and denote this set by 
$\mathcal{R}_\Delta$.  By topological transitivity there exists $N\in \N$ such that $f^N(X)=I$, 
and since Hausdorff dimension is preserved by bi-Lipschitz mappings, Theorem~\ref{thm:lift} 
says $\dim_H\left(f^N(\nolift_\Delta)\right)\le D$, where $\dim_H(\cdot)$ denotes the Hausdorff dimension of a set.  
Therefore, the set of points in $I$ which are not liftable must have Hausdorff 
dimension $\le D$.  This means that if we were interested in a set  $A\subset I$ 
that has Hausdorff dimension greater than $D$ then it must be `seen by the tower': 
$$\dim_H\left(A\cap \mathcal{R}_\Delta\right)=\dim_H(A).$$

We will use this information on dimension in conjunction with measures.
Setting $\tau(x) = \infty$ for all points for which $\tau(x)$ is not originally
defined, we say that a measure 
$\mu$ on $I$ \emph{lifts} to the inducing scheme 
(correspondingly lifts to the tower $\Delta$) if $\mu(X)>0$ and $\mu(\tau)<\infty$. 
If $\mu$ is an $f$-invariant measure which lifts, then according to
\cite[Theorem 1.1]{zweimuller},
there is an \emph{induced measure} $\nu \ll \mu$ supported on $X$
which is $F$-invariant and such that for any $A\subset I$, 
\begin{equation}
\mu(A)=\frac1{\int_X \tau~d\nu}\sum_i\sum_{k=0}^{\tau_i-1}\nu(f^{-k}(A)\cap X_i)
=\frac1{\int_X \tau~d\nu} \sum_{k\ge 0} \nu(\{ \tau > k \} \cap f^{-k}(A)).
\label{eq:proj meas}
\end{equation}
Conversely, given an $F$-invariant measure $\nu$ such that $\int\tau~d\nu<\infty$, 
there exists an $f$-invariant measure $\mu$ defined by \eqref{eq:proj meas} and we 
say that $\nu$ \emph{projects} to $\mu$.  

We can also consider the intermediate 
measures on the Young tower here: given an $F$-invariant measure $\nu$ with 
$\int\tau~d\nu<\infty$, we can define a measure $\mu_\Delta'$ on $\Delta$ by 
putting $\mu_\Delta'|_{\Delta_0}$ to simply be a copy of $\nu$ and then for any 
$A\subset \Delta_{\ell,i}$ where $0\le \ell\le \tau_i-1$, since $f_\Delta^{-\ell}(A)\subset \Delta_0$, we can set $\mu_\Delta'(A)$ to be $\mu_\Delta'(f_\Delta^{-\ell}(A))$.  Then we obtain $\mu_\Delta:=\frac1{\int\tau~d\nu}\mu_\Delta'$.  Observe that $\mu$ as in \eqref{eq:proj meas} is now the push-forward of $\mu_\Delta$ by $\pi$.

We set the pressure of $\phi$ on $\Delta$ to be
$$
P_\Delta(\phi):= \sup_{\mu \in \M_f} \left\{h(\mu)+\int\phi~d\mu :\mu \text{ lifts to } \Delta \text{ and } \mu(-\phi) < \infty \right\}.
$$
Clearly $P_{\M_f}(\phi)\ge P_\Delta(\phi)$ for any $\phi$.  Letting $\M_F$ be the set of $F$-invariant probably measures, we can define the pressure $P_{\M_F}$ as usual.  To see the relation between these different pressures, for a potential $\phi:I\to [-\infty,\infty]$, define 
$$S_n\phi = \sum_{k=0}^{n-1} \phi \circ f^k.$$
Then we obtain the \emph{induced potential} $\Phi(x)$ as $S_{\tau_i}\phi(x)$ for $x\in X_i$.
Abramov's formula says that if $\mu$ lifts to $\nu$ then
$$h(\mu)=\frac{h(\nu)}{\int\tau~d\nu} \quad \text{ and } \quad \int\phi~d\mu = \frac{\int\Phi~d\nu}{\int\tau~d\nu}.$$

This also holds for measures on $\Delta$, since any invariant probability 
measure $\mu_\Delta$ on $\Delta$ must lift to the corresponding inducing scheme; 
indeed the lifted measure is simply 
$\nu=\frac{\mu_\Delta|_{\Delta_0}}{\mu_\Delta(\Delta_0)}$, 
the conditional measure on $\Delta_0$.
The above relations then hold due to
the extra information provided by Kac's Lemma that 
$\int_{\Delta_0} \tau~d\nu =\frac1{\mu_\Delta(\Delta_0)}$.

We inductively define $\tau^n(x) 
=\tau^{n-1}(x)+\tau(f^{\tau^{n-1}(x)}(x))$ to
be the $n$th return of $x$ to $X$ under the inducing scheme, 
and let $\{X_{k,i}\}_i$ be the set of $k$-cylinders
(cylinders of length $k$) for $(X,F)$: that is, for each $i$, $\tau^k$ is constant on $X_{k, i}$ and $F^k:X_{k,i}\to X$ is a diffeomorphism.
Moreover, for a potential $\Phi:\cup_iX_i\to \R$, define the \emph{$n$-th variation} of $\Phi$ as
$$V_n(\Phi):=\sup\{|\Phi(x)-\Phi(y)|:x, y \text{ are contained in the same } n\text{-cylinder}\}.$$
Then $\Phi$ is said to be \emph{locally H\"older} if there exists $\eta>0$ such that $V_n(\Phi)=O(e^{-\eta n})$ for all $n \in \N$.  This condition on potentials is required to define the following kind of measure.
We say that $\nu$ is a Gibbs measure for the potential $\Phi$ 
if $\nu(-\Phi)< \infty$ and
there exist constants $C>0$, $P \in \mathbb{R}$, such that for each 
$n$-cylinder $X_{n,i}$ with respect to $F$, we have
\[
C^{-1} e^{S_n\Phi(y_i)-nP} \le \nu(X_{n,i}) \le C e^{S_n\Phi(y_i)-nP}
\]
for any $y_i \in X_{n,i}$, where $S_n\Phi = \sum_{k=0}^{n-1} \Phi \circ F^k = 
\sum_{k=0}^{\tau^n-1} \phi \circ f^k$.  
Note that the induced potentials considered in this paper will be locally H\"older and have 
equilibrium states which are Gibbs measures with the constant $P$ above equal to the pressure.

\begin{remark}
By \cite{Prz93}, given $f$ as above, any $\mu\in \M_f$ has $\int\log|Df|~d\mu\in [0, |Df|_\infty]$.  Therefore our potentials $\vf_t$ all have $\int\vf_t~d\mu$ finite. Thus we can drop the condition $-\int\vf_t~d\mu<\infty$ in the definition of pressure $P_{\M_f}(\vf_t)$.  Moreover,  since as described above any $f_\Delta$-invariant measure $\mu_\Delta$ with $\mu_\Delta(\Delta)=1$ projects to a measure $\mu\in \M_f$, the same conclusion can be drawn for measures on $\Delta$
(see, in particular, Lemma~\ref{lem:project jac}).  Notice, however, that  for any of our inducing schemes $(X, F, \tau)$, there are measures $\mu_F\in \M_F$ that have $\int\log|DF|~d\mu_F=\infty$.  While we need to keep these in mind when computing $P_{\M_F}$, these measures will not to be relevant here.
\end{remark}


\section{Main results}
\label{sec:results}

We fix a set $\Crit_{\mbox{\tiny hole}}$ as described in Section~\ref{ssec:intro holes}
and $\delta_0 >0$ sufficiently small
such that (H2) and \eqref{eq:no holes mixing} are satisfied for $\delta = \delta_0$.  This in turn 
fixes all $\delta$ appearing in (C1)-(C2) to have value $\delta_0$.  

For this choice of $\Crit_{\mbox{\tiny hole}}$, for $h>0$, 
define $\H(h)$ to be the family of holes $H \subset I$
such that 
\begin{enumerate}
  \item $b_j \in H_j$ and $m_1(H_j) \le h$ for each $j = 1, \ldots, L$;
  \item $H$ satisfies (H1),
\end{enumerate}
where $m_1$ denotes Lebesgue measure (the $\vf_1$-conformal measure).

\begin{theorem}
Let $f$ satisfy (C1)-(C2) and fix the set $\Crit_{\mbox{\tiny hole}}$ as above.   
There exist
constants $0<t_0 <1<t_1$ and $h >0$ such that if $t  \in [t_0, t_1]$ and 
$H \in \H(h)$, then $f$ admits an inducing scheme $(X, F,\tau, H)$ 
that respects the hole and 
\begin{enumerate}
\item[(a)] there exists $D<1$ such that 
$\dim_H(\nolift_\Delta)\le D$;
\item[(b)] $\dim_H(\hDelta^\infty)=\dim_H(\I^\infty)$;
\item[(c)] $t\in [t_0, t_1]$ implies that $m_t(\tau<\infty)=m_t(X)$;
\item[(d)] if $t\in [t_0,t_1]$ then any measure $\mu$ on $I^\infty$ which doesn't lift to $(\Delta, f_\Delta)$ must have $h(\mu)+\int\varphi_t^H~d\mu<P_\Delta(\varphi_t^H)$.
\end{enumerate}
\label{thm:lift}
\end{theorem}

We remark that (a), which is the solution of a problem of the liftability of measures, was used in \cite[Theorem 7.6]{PesSen08} applying results of \cite{Sen03}.
As can be seen from the proof, the constant $D$ there is of the form 
$\log K/(\alpha_1 + \log K)$ where $\alpha_1$ is the exponential return rate in the 
inducing scheme and $K$ depends on the complexity of the map, including 
the placement of the hole.

\begin{proposition}
\label{prop:uniform tails}
Let $f$, $t_0<t_1$, $H$ and $\Delta$ be as in Theorem~\ref{thm:lift}.  Then there exist constants $C_0, \alpha>0$, depending
only on $t_0, t_1$ and $h$ such that for each $t\in [t_0, t_1]$ we have $m_t(\tau >n) \le C_0e^{-\alpha n}$, for all $n\ge 0$.
\end{proposition}

We denote by $\cont^p(I)$ the space of H\"older continuous functions on $I$
with exponent $p$.  In what follows, we assume that $p \ge \beta/\log \xi$, where
$\xi >1$ is defined in (A1) in Section~\ref{ssec:prop proof} and $0<\beta < \alpha$ from
Section~\ref{sec:LY} defines
the metric on $\Delta$.\footnote{In fact, $\beta$ can be chosen as small as one likes,
allowing $p$ to be arbitrarily small.  The price to pay is that then $h$ must be small
according to (P4) of Section~\ref{sec:LY}.} 
For relevant $t$, let
$g^0_t$ denote the density of the equilibrium measure $\mu_t$ for $f$ with respect to $m_t$.
Note that $g^0_t \notin \cont^p(I)$ due to spikes corresponding to the critical orbits.

\begin{theorem}
\label{thm:accim}
Let $f$, $t_0<t_1$ and $H$ be as in Theorem~\ref{thm:lift} and recall that $\phi_t = t\vf - p_t$.  
Then for each $t \in [t_0, t_1]$,
$\Lp_{\phi_t^H}$ has a unique simple eigenvalue $\lambda^H_t \le 1$ of 
maximum modulus whose corresponding eigenvector $g_t^H$ defines an absolutely continuous
conditionally invariant measure $\mu^H_t = g_t^H m_t$.  

In addition, there exist constants $C_* >0$, $\sigma <1$ such that for any $\psi \in \cont^p(I)$ with
$\psi>0$,
\beq 
\label{eq:limit}
\lim_{n \to \infty} \left| \frac{\Lp_{\phi_t^H}^n \psi}{| \Lp_{\phi_t^H}^n \psi |_{L^1(m_t)}}  - g_t^H \right|_{L^1(m_t)} \le C_* \sigma^n | \psi |_{\cont^p(I)}  \; \; \; \mbox{for all } n \geq 0.
\eeq
The same limit holds for the sequence $\frac{\Lp_{\phi_t^H}^n (\psi g^0_t)}{| \Lp_{\phi_t^H}^n (\psi g^0_t) |_{L^1(m_t)}}$ and all positive $\psi \in \cont^p(I)$.
In particular $\e(\psi m_t) = \e(\psi \mu_t) = -\log \lambda^H_t$
for all $\psi \in \cont^p(I)$, $\psi >0$.
\end{theorem}

In fact, the limit \eqref{eq:limit} holds for any $\psi$ which projects down from the 
relevant function space on the tower.  Since the result in this generality is technical to state,
we refer the reader to Proposition~\ref{prop:conv} and Remark~\ref{remark:conv}
for more precise statements.

\begin{theorem}
\label{thm:variational}
Under the hypotheses of Theorem~\ref{thm:accim}, there exist $\bar t_0\in [t_0,1)$ 
and $\bar t_1\in (1,  t_1]$, such that for all $ t\in [\bar t_0, \bar t_1]$, 
\begin{equation}
\label{eq:var}
-\e(m_t) = \sup_{\mu \in \G_f^H} \left\{ h_\mu(f) + t\int \vf d\mu \right\} - P_{\M_f}(t\vf) . 
\end{equation}
Moreover, the following limit defines a measure $\nu_t^H$,
\[
\nu_t^H(\psi) = \lim_{n \to \infty} e^{\e(m_t)n} \int_{\I^n} \psi g_t^H dm_t, \; \; \; \mbox{for all }
\psi \in \cont^0(I).
\]
The measure $\nu_t^H$ belongs to $\G_f^H$ and 
attains the supremum in \eqref{eq:var}.
In addition, the limit \eqref{eq:limit} holds for any $\psi \in \cont^p(I)$ with 
$\nu_t^H(\psi) > 0$, and for $\psi g^0_t$ whenever $\nu_t^H(\psi g^0_t) >0$.
\end{theorem}

The next theorem characterizes the Hausdorff dimension of the survivor set according to
the Bowen formula.  

\begin{theorem}
\label{thm:hausdorff}
Under the hypotheses of Theorem~\ref{thm:variational}, 
dim$_H(\I^\infty) =$ dim$_H(\hDelta^\infty) = t^*$, where $t^*<1$ is the unique value of $t$
such that $\sup_{\mu \in \G_f^H} \{ h_\mu(f) + t \int \vf d\mu : - \mu(\vf) < \infty \} = 0$,
\ie such that $\e(m_t) = P_{\M_f}(t\vf)$. 
\end{theorem}

To state our final results regarding the zero-hole limit, in what follows we will
take our holes
to be symmetric intervals 
$$H_\eps=H_\eps(z):=(z-\eps, z+\eps)$$  
around a point $z\in I$.  
To state our next result, in the case that $z$ is a periodic point, 
we will need to make an assumption on $z$, which we call condition 

\begin{enumerate}
  \item[{\bf (P)}] The density $\frac{d\mu_t}{dm_t}$ is bounded at $z$, for the relevant $t$, and 
  condition \eqref{eq:slow approach} given in Section~\ref{zero hole} is satisfied.
\end{enumerate}

\begin{theorem} 
Let $t \in [t_0, t_1]$, $z\in I$ and suppose $\Crit_{\mbox{\tiny hole}} := \{ z \}$ satisfies (H2).
Fix $\delta = \delta_0 > 0$ appearing in (C1)-(C2).
Then there exists $h >0$ such that
\begin{equation*}
\frac{\e(m_t, H_\ve(z))}{\mu_t(H_\ve(z))}\to \begin{cases} 1 & \text{ for } \mu_t\text{-a.e. } z,\\
 1-e^{S_p\phi_t(z)} & \text{ if } z \text{ is periodic of (prime) period } p \text{ and (P) holds},
 \end{cases}
 \end{equation*}
 where the limit is taken as $\ve \to 0$ over $H_\ve(z) \in \H(h)$.
 \label{thm:zero hole}
 \end{theorem}

In the following result, we show that our conditions for Theorem~\ref{thm:zero hole}
are met in a reasonable family of maps.
Consider the logistic family $f_\lambda:x\mapsto \lambda x(1-x)$ on $I$.  
Recall that if $f_\lambda$ has a hyperbolic periodic cycle 
$\{z_\lambda, f_\lambda(z_\lambda), \ldots, f_\lambda^{n-1}(z_\lambda)\}$, 
then for all nearby $\lambda'$, this cycle persists in the sense that there is a cycle 
$\{z_{\lambda'}, f_{\lambda'}(z_{\lambda'}), \ldots, f_{\lambda'}^{n-1}(z_{\lambda'})\}$ 
which depends analytically on the parameter and which converges to the original 
cycle as $\lambda'\to \lambda$.  This family of cycles is called the 
\emph{hyperbolic continuation} of the original cycle.

For this family of maps, we denote by $m_{\lambda,t}$ the $(\vf_t - p_t)$-conformal measure
for $f_\lambda$ and by $\mu_{\lambda,t} \ll m_{\lambda,t}$ its equilibrium state,
when it exists.
 
\begin{theorem}
There is a positive measure set of parameters $\Omega\subset (3,4]$ and an interval $[t_0, t_1]\ni1$ such that for $t \in [t_0,t_1]$ and $\lambda \in \Omega$, the map $f_\lambda$ 
has an equilibrium state $\mu_{\lambda, t}\ll m_{\lambda, t}$,  and
\begin{itemize}
\item[(a)] for $\mu_{\lambda, t}$-a.e. $z\in I$
$$
\frac{\e(m_{\lambda, t}, H_\ve(z))}{\mu_{\lambda, t}(H_\ve(z))}\to 1;$$
\item[(b)] if $z_4$ is a periodic point of (prime) period $p$ for $f_4$, then there exists a positive measure family of parameters $\Omega(z_4)\subset \Omega$ such that for $\lambda \in \Omega(z_4)$ and for  $z_\lambda$ denoting the hyperbolic continuation of $z_4$, 
$$
\frac{\e(m_{\lambda, t}, H_\ve(z_\lambda))}{\mu_{\lambda, t}(H_\ve(z_\lambda))}\to 1-e^{S_p\phi_t(z_\lambda)}.$$
\end{itemize}
In both limits above, (H2) is assumed to hold for $\Crit_{\mbox{\tiny hole}} = \{ z \}$, and
the limit is taken as $\ve \to 0$ over holes $H_\ve(z) \in \H(h)$, 
for $h>0$ sufficiently small.
  \label{thm:zero hole adapted}
\end{theorem}


\section{An inducing scheme with uniform tail rates which covers most of our space: Theorem~\ref{thm:lift} and  Proposition~\ref{prop:uniform tails}}
\label{sec:tails}

In this section we describe our inducing schemes and then prove that the Hausdorff dimension of points not liftable to $\Delta$ is not large: This is Theorem~\ref{thm:lift}(a).  In order to prove the remaining parts of that theorem, it is useful to have some continuity properties of Hausdorff dimension of measures and of the sets $\hDelta^\infty$.  Since Proposition~\ref{prop:uniform tails} gives us some of these properties we prove the proposition before completing the proof of Theorem~\ref{thm:lift}.

As in Section~\ref{sec:results}, we fix a set $\Crit_{\mbox{\tiny hole}}$ and 
$\delta_0 >0$ sufficiently small
such that (H2) and \eqref{eq:no holes mixing} are satisfied for $\delta = \delta_0$.  This in turn 
fixes all $\delta$ appearing in (C1)-(C2) to have value $\delta_0$.  
Let $\H(h)$ denote the family of holes defined in Section~\ref{sec:results}.

Under these conditions, in \cite{BDM} an interval $X \subset B_{\delta_0}(c)$ for some
$c \in \Crit_c$  is fixed and
inducing schemes $(X,F, \tau, H)$ are constructed
over $X$
with uniform tails for all  $H \in \H(h)$:  there exists
$C_1, \alpha_1>0$ such that $m_1(\tau > n) \le C_1 e^{-\alpha_1 n}$ where 
$C_1$ and $\alpha_1$
depend only on $h$ once $\Crit_{\mbox{\tiny hole}}$ is fixed.
Necessarily, $h \le \delta_0$, so that $X$ is disjoint from $H$ for all $H \in \H(h)$.
We set $\delta_1 = m_1(X)$ and note that $\delta_1 < \delta_0$ by
the construction in \cite{BDM}.

\subsection{Proof of Theorem~\ref{thm:lift}(a)}
\label{hausdorff}

The proof of Theorem~\ref{thm:lift}(a) follows \cite[Lemma 4.5]{BDM} closely 
(which in turn follows \cite{DHL} rather closely).

The inducing time on the interval $X$ of length $\delta_1 < \delta_0$ is constructed 
following a standard
algorithm: $X$ is iterated forward until it is cut by either the boundary of the hole or the exponential
{\em critical partition} defined in a $\delta_0$-neighborhood of each critical point.  
The resulting subintervals 
are then iterated, waiting for a proper return to $X$ which occurs during a `free' period and 
consists of an interval which overlaps $X$ by at least $\delta_1/3$ on each side.
This defines the return time function $\tau$ on $X$.
Intervals passing through 
$B_{\delta_0}(c)$ for some $c \in \Crit_c$ 
undergo a `bound' period whose length depends on
the {\em depth} of the return, which is the index of the critical partition.  
The definitions of `free' and `bound' in this context are by now standard
(see \cite[Sect. 2.2]{DHL}).  The only new feature created after the introduction of the hole
in \cite[Lemma 4.5]{BDM} is a short bound period which allows the derivative to grow
sufficiently between cuts due to $\partial H$.  This avoids the problem of repeated cutting
potentially destroying expansion and is formulated formally in (H2)(c). 

Let $\Q^{(n)}$ denote the set of subintervals of $X$ induced by these
subdivisions after $n$ steps which have not made a proper return by time $n$.
By \cite[Lemma 4.5]{BDM}, we have
\begin{equation}
\label{eq:Q}
\sum_{\omega \in \Q^{(n)}} |\omega| \le C_1 e^{-\alpha_1 n} ,
\end{equation}
where $|\omega|$ denotes the length of the interval $\omega$.

The set of intervals in $\Q^{(n)}$ form an open cover  for $\nolift_\Delta$ for each $n$, 
thus we may
use them to bound the Hausdorff dimension of $\nolift_\Delta$.  First we 
define a coarser partition $\tilde\Q^{(n)}$ of the set $\{ \tau > n \}$ by grouping
intervals $\omega \in \Q^{(n)}$ as follows:
We glue together
adjacent intervals $\omega \in \Q^{(n)}$ which are in a  bound period at time $n$
and which have not been separated by an intersection with 
$\partial H \cup \Crit_c \cup (\cup_{c \in \Crit_c} \partial B_{\delta_0}(c))$ at any
time $0 \le k \le n$.

We claim that the cardinality of $\tilde \Q^{(n)}$ is finite and bounded exponentially
in $n$.  To prove this claim, note that $\tilde \Q^{(1)} = \{ X \}$ and
that elements of $\tilde\Q^{(n)}$ are formed from elements $\omega \in \tilde \Q^{(n-1)}$
in one of three ways:
\begin{itemize}
  \item $f(\omega) \cap \partial H \neq \emptyset$;
  \item $f(\omega) \cap (c \cup \partial B_{\delta_0}(c))$  for some $c \in \Crit_c$;
  \item $\omega$ is bound at time $n-1$, but part of $\omega$ becomes free
  at time $n$.
\end{itemize}
Since $\partial H$ and $\Crit_c$ are finite, the only point we need to consider is the third.

Suppose $x$ enters a bound period in $B_{\delta_0}(c)$ at time $k$ with depth $r$ in the critical 
partition, \ie $|f^k(x) - c| \approx e^{-r}$.  In order for $x$ to become free again at time
$n$, we must have $|(f^{n-k})'(x)| \ge \kappa^{-1} e^{\theta (n-k)}$ for some constant 
$\theta >0$ by
\cite[Lemma 2]{DHL}.  By (C1), we must also have $|(f^k)'(x)| \ge \kappa e^{\gamma k}$
upon entry to $B_{\delta_0}(c)$.
Thus $|(f^n)'(x)| \ge \zeta$ for some constant $\zeta >1$.
Now let $M = \max_{x\in I} |Df(x)|$.  
Then since
$|Df(f^kx)| \le C e^{-(\ell_c-1)r}$ by definition of the critical order $\ell_c$ of $c$, we must have
\[
\zeta \le |Df^n(x)| \le C e^{-(\ell_c -1)r} M^n \implies r \leq \frac{n \log M - \log (\zeta/C)}{\ell_{\max} -1} .
\]
So the number of intervals becoming free at time $n$ from a single intersection  with
the critical partition has a linear bound in $n$.  Putting these facts together with the
finiteness of $\partial H$ and $\Crit_c$ implies that the cardinality of
$\tilde \Q^{(n)} \le C K^n$ for some constants $C>0$, $K >1$ and all $n \in \N$ as required.

Now for $s<1$, we estimate using \eqref{eq:Q},
\[
\begin{split}
\sum_{\tilde \omega \in \tilde \Q^{(n)}} |\tilde \omega|^s
& \le \Big( \sum_{\tilde \omega \in \tilde \Q^{(n)}} |\tilde \omega| \Big)^s (\# \tilde Q^{(n)})^{1-s} \\
& = \Big( \sum_{\omega \in \Q^{(n)}} |\omega| \Big)^s (\# \tilde Q^{(n)})^{1-s}
\le C_1^s e^{-\alpha_1 n s} C^{1-s} K^{n(1-s)} .
\end{split}
\]
The above expression tends to 0 as $n \to \infty$ as long as
$- \alpha_1 s + (1-s) \log K <0$, \ie as long as $s > \log K/(\alpha_1 + \log K)$.
Thus $\dim_H(\nolift_\Delta) \le \log K/(\alpha_1 + \log K) =: D <1$, and Theorem~\ref{thm:lift}(a) is proved.

\subsection{Proof of Proposition~\ref{prop:uniform tails}}
\label{ssec:prop proof}

For our class of $\cont^2$ maps, the inducing schemes and towers constructed 
in \cite[Section 4.3]{BDM} satisfy the 
following properties.  For the distortion bounds, see also \cite[Proposition 3]{DHL}.

\medskip
\noindent
\parbox{.07 \textwidth}{\bf(A1)}
\parbox[t]{.9 \textwidth}{ There exist constants $\xi > 1$ and $C_d' > 0$ such that
\begin{itemize}
  \item[(a)] for any $x \in X$, $n \geq 1$ and $k < \tau^n(x)$,
  $|Df^{\tau^n(x)-k}(f^kx)| \ge (C_d')^{-1} \xi^{\tau^n(x)-k}$.
  \item[(b)] Let $x,y \in X_i$ and $\tau_i = \tau(X_i)$.  Then
  $\left| \frac{e^{S_\ell \varphi(x)}}{e^{S_\ell \varphi(y)}} \right| \leq C_d'$ for $\ell \leq \tau_i$.    
  In addition,
  $\left| \frac{e^{S_{\tau_i}\varphi(x)}}{e^{S_{\tau_i}\varphi(y)}} -1 \right|
        \leq C_d' d(f^{\tau_i}(x),f^{\tau_i}(y))$, or in alternative notation,
        $\left| \frac{e^{\Phi(x)}}{e^{\Phi(y)}} - 1 \right| \le C_d' d(F(x), F(y))$,
        where $\Phi$ is the induced version of $\vf$.
\end{itemize} }

\medskip
\noindent
\parbox{.1 \textwidth}{\bf(A2)}
\parbox[t]{.85 \textwidth}{There exists $L<\infty$ and an index set  
$\mathcal{I} \subset [0,L] \times \mathbb{N}$ such that\footnotemark
\begin{itemize}
  \item[(a)]  $m_1(I \backslash \cup_{(\ell, j) \in \mathcal{I}} \pi(\hDelta_{\ell,j})) = 0$;
  \item[(b)]  $\pi(\hDelta_{\ell_1, j_1}) \cap \pi(\hDelta_{\ell_2, j_2}) = \emptyset$ for all but finitely many
    $(\ell_1, j_1)$, $(\ell_2, j_2) \in \mathcal{I}$;
  \item[(c)]  Define $J_1\pi_{\ell, j} := J_1\pi|_{\hDelta_{\ell,j}}$.
  Then $\sup_{(\ell,j) \in \mathcal{I}} |J_1\pi_{\ell,j}|_\infty +  \mbox{Lip}(J_1\pi_{\ell,j}) < \infty$.
\end{itemize}  }
\footnotetext{See \cite[Lemma 4.6]{BDM} for a proof of property (A2).}

Here $J_1\pi$ is the Jacobian of $\pi$ with respect to Lebesgue measure $m_1$ and 
the corresponding induced measure $\bm_1$ on $\Delta$ and 
$\mbox{Lip}$ denotes the Lipschitz constant measured in the symbolic metric $d_\beta$
on $\Delta$ (both $\bm_1$ and $d_\beta$ are defined in Section~\ref{sec:LY}).
Property (A1) guarantees expansion and bounded distortion at return times.  
In particular, it guarantees that the partition $\{ \Delta_{\ell,j} \}$ is generating.
(A2) ensures that the covering of $I$ by $\hDelta$ is sufficiently nice
that all smooth functions can be realized as projections of Lipschitz densities from the tower
\cite[Proposition 4.2]{BDM}.

\begin{lemma}[Bounded Distortion for $\phi_t$]
\label{lem:distortion}
For the potential $\phi_t(x):=-t\log|Df(x)|-p_t$, we have the induced potential 
$\Phi_t(x)  -\tau(x)p_t= S_{\tau(x)}\vf_t(x)-\tau(x)p_t$ for the inducing scheme.  There exists $C_d >0$ such that for all $t \in [0,2]$ and any $x, y \in X_i$, for $0\le \ell\le \tau_i$, 
we have
$\left| \frac{e^{S_\ell \phi_t(x)}}{e^{S_\ell \phi_t(y)}} \right| \le C_d$.  Moreover, 
\[
\left| \frac{e^{ \Phi_t(x)-\tau(x)p_t}}{e^{\Phi_t(y)-\tau(y)p_t}} -1 \right| \le C_d d(F(x), F(y)).
\] 
Furthermore, $\Phi_t-\tau p_t$ is locally H\"older.  
\end{lemma}

\begin{proof}
Since $e^{S_\ell \phi_t(x)} = e^{t S_\ell \vf(x) - \ell p_t}$ and $e^{S_\ell \phi_t(y)} = e^{t S_\ell \vf(y) - \ell p_t}$, (A1)(b) implies immediately that $\left| \frac{e^{S_\ell \phi_t(x)}}{e^{S_\ell \phi_t(y)}} \right| \le (C_d')^t$.
For the second equation, it is equivalent to estimate
\[
\log \left| \frac{e^{S_{\tau_i} \phi_t(x)}}{e^{S_{\tau_i} \phi_t(y)}} \right|
\le t \log \left| \frac{e^{S_{\tau_i} \vf(x)}}{e^{S_{\tau_i} \vf(y)}} \right|,
\]
so this part of the lemma follows from (A1)(b) as long as $t$ is in a compact interval.  This implies that $V_1(\Phi_t)$ is bounded.  The fact that $V_n(\Phi_t)$, and hence $V_n(\Phi_t-\tau p_t)$, decays exponentially in $n$ follows similarly, but now also using (A1)(a) to show that if $x, y$ are in the same $n$-cylinder, then $d(F(x), F(y))$ is exponentially small in $n$.
\end{proof}

\begin{remark}
We have already stated that for the potential $\varphi_t(x):=-t\log|Df(x)|$, we have the induced potential 
$\Phi_t(x) := S_{\tau(x)}\vf_t(x)$ for the inducing scheme.  Now we also denote the 
corresponding punctured potentials by $\varphi_t^H$ and $\Phi_t^H$. 
We often denote $\Phi_1$, the induced potential corresponding to $\varphi$,  by $\Phi$.  
Observe that since our inducing schemes respect the hole, any measure on the survivor set cannot give mass to any column $\cup_{\ell=0}^{\tau_i-1}\Delta_{i, \ell}$ with a hole in it.  Therefore, for the purposes of computing pressure, we can either think of $\Phi_t^H$ on the corresponding base element $X_i$ as being $-\infty$, or think of $X_i$ as not being part of the inducing scheme.  In either case, the H\"older property of the variations persists for the punctured potentials.
\label{rmk:holder punc}
\end{remark}

The tail estimates in \cite{BDM} hold for the inducing schemes with respect to Lebesgue
measure.
We will show that we have related tail rates for the conformal measures corresponding to $\Phi_t$ for all $t$ sufficiently close to 1 and for all small holes.

Define 
\beq
\begin{split}
\chi_M & :=\sup\{\log|Df^n(x)|/n: f^n(x)=x\} \text{ and }\\
\chi_m & :=\inf\{\log|Df^n(x)|/n: f^n(x)=x\}.
\end{split}
\label{eq:chi}
\eeq
By \cite{BrvS}, (C2) implies that $\chi_m>0$.  Moreover,
for a measure $\mu\in \M_f$, let $$\chi(\mu):=\int\log|Df|~d\mu.$$  

Each interval $X_i$ contains a point $x_i$ such that $f^{\tau_i}(x_i) = x_i$.  
So by the
bounded distortion of Lemma~\ref{lem:distortion} and the definitions of $\chi_m$ and $\chi_M$,
we have $- \chi_M \tau_i \lesssim S_{\tau_i} \vf(y_i)= \Phi(y_i) \lesssim - \chi_m \tau_i$, 
for any $y_i \in X_i$, where $\lesssim$ denotes a uniform constant depending only on the distortion.  
Now $S_\tau \phi_t = t S_\tau \vf - \tau p_t$.  
Therefore, choosing
$y_i' \in X_i$ such that $e^{S_{\tau_i} \vf(y_i')} = m_1(X_i)/m_1(X)$, we
obtain for $0 < t < 1$,
\[
\sum_{\tau_i = n} e^{S_\tau \phi_t(y_i)}
\le C_d e^{n[(1-t)\chi_M - p_t]} \sum_{\tau_i = n} e^{S_{\tau_i} \vf(y_i')} 
\le C_1 C_d \delta_1^{-1} e^{n[(1-t)\chi_M - p_t]} e^{-\alpha_1 n}.
\]
Now since $p_1=0$ and $p_t$ is a continuous function of $t$ (see for example \cite{ITeq}), we may choose 
$t_0'<1$ such that $(1-t)\chi_M - p_t - \alpha_1 < 0$ for all $t \in [t_0', 1]$.   

The proof for $t>1$ follows similarly with $\chi_m$ in place of $\chi_M$.  In this case,
the constraint is $-(t-1)\chi_m - p_t - \alpha_1 < 0$ for all $t \in [1, t_1']$, with the 
restriction on $t_1'>1$ coming from the fact that $p_t<0$.
Then define 
$$
\alpha := \alpha_1 - \max \left\{ \sup_{t \in [t_0',1]} (1-t)\chi_M - p_t, \sup_{t \in [1, t_1']} (1-t) \chi_m - p_t \right\} .
$$

By conformality and bounded distortion, we have 
$m_t(X_i) = m_t(X) e^{S_\tau \phi_t(y_i)}$ for some $y_i \in X_i$ and each $i$,
which completes the proof of
Proposition~\ref{prop:uniform tails}.

Note that in order for Proposition~\ref{prop:uniform tails} and 
Theorem~\ref{thm:lift} to be coherent,  we 
will choose the interval $[t_0,t_1]$ in the statement of the Proposition~\ref{prop:uniform tails} to be 
contained in $[t_0', t_1']$ chosen above in order that the Hausdorff dimension considerations 
discussed in Section~\ref{hausdorff} hold.


\subsection{Proof of parts (b)--(d) of Theorem~\ref{thm:lift}}

For this proof we will repeatedly make use of the fact, proved in \cite{Hofdim}, that if $\mu\in \M_f$ then $\dim_H(\mu)=h(\mu)/\chi(\mu)$.

For part (b), we write $\I^\infty \subset \pi(\hDelta^\infty) \cup \nolift_\Delta$ and by
Theorem~\ref{thm:lift}(a), we have $\dim_H(\nolift_\Delta) \le D< 1$.
Note that the invariant measure $\nu_1^H$ corresponding to the potential $\vf_1$ from 
\cite{BDM} satisfies $\log \lambda_1^H = h(\nu_1^H) - \chi(\nu_1^H)$, so that
$\dim_H(\nu_1^H) = \frac{h(\nu_1^H)}{\chi(\nu_1^H)} = 1 + \log \lambda_1^H$.  
Moreover, in \cite{BDM} it is shown that $ \log \lambda_1^H\to 0$ as the hole shrinks to zero, 
so $\dim_H(\nu_1^H)$ can be
made arbitrarily close to 1 (and $>D$) for any $H \in \H(h)$ with $h$ sufficiently small.
This implies that $\dim_H(\pi(\hDelta^\infty)) > D$ so that
necessarily, $\dim_H(\I^\infty) = \dim_H(\hDelta^\infty)$.

As before, let $\mu_t$ denote the equilibrium state for $\vf_t$ (before the introduction of the hole).
For (c), we have 
$\dim_H(m_t)=\dim_H(\mu_t)=\frac{h(\mu_t)}{\chi(\mu_t)}=t+\frac{P_{\M_f}(\vf_t)}{\chi(\mu_t)}$.   
As in \cite{ITeq}, this value is continuous in $t$, so $\dim_H(m_t)$ is close to 1 for $t$ 
close enough to 1 since $P_{\M_f}(\vf_1)=0$.  So 
part (a) implies that $m_t(\nolift_\Delta) = 0$.  
We set $t_1' \ge t_1>1$ and $t_0'\le t_0<1$ to be such that this holds for all $t\in [t_0, t_1]$.

For (d), we will use the fact that $\chi_m>0$.  
Part (a) implies that a measure $\mu\in \M_f^H$ which doesn't lift to  $(\Delta, f_\Delta)$ must have 
$\dim_H(\mu)=h(\mu)/\chi(\mu)<D$.  Thus 
$$
h(\mu) + t\int\vf^H~d\mu=
h(\mu)-t\chi(\mu)<-(t-D)\chi(\mu)\le -(t-D)\chi_m.
$$
Since for $t$ sufficiently close to 1 and $H$ sufficiently small, 
$P_\Delta(\varphi_t^H)$ is approximately 0, the result follows. 
Note that this holds since $P_\Delta(\varphi_1^H)$ is continuous in the size of the 
hole by \cite{BDM} and $P_\Delta(\varphi_t^H)$ is continuous in $t$.


\section{Review of Known Results: Transfer Operator on the Tower with Holes}
\label{sec:LY}

We recall the abstract setup of \cite{BDM} into which we shall place our induced maps
in order to prove Theorems \ref{thm:accim} -- \ref{thm:zero hole}.
 
Let $f_\Delta : \Delta \circlearrowleft$ be a Young tower 
formed over an inducing scheme $(f,X, \tau)$ as described in Section~\ref{ssec:induced}.
Given a $\phi$-conformal reference measure $m$ on $I$,
we define a reference measure $\bm$ on $\Delta$ by $\bm = m$ on $\Delta_0=X$ and
$\bm|_{\Delta_\ell} = (f_\Delta)_* \bm|_{\Delta_{\ell-1} \cap f_\Delta^{-1} \Delta_\ell}$ 
for $\ell \ge 1$.  
For $x \in \Delta_\ell$, let $x^- := f^{-\ell}x$ 
denote the pullback of $x$ to $\Delta_0$.
We define the induced
potential on $\Delta$ by,
\begin{equation}
\label{eq:Delta potential}
\phi_\Delta(x) = S_\tau \phi (x^-) \; \; \mbox{for } x \in f_\Delta^{-1}(\Delta_0)
\; \; \mbox{and} \; \; \phi_\Delta = 0 \; \; \mbox{on } \Delta \setminus f_\Delta^{-1}(\Delta_0) .
\end{equation}
With these definitions, the measure $\bm$ is $\phi_\Delta$-conformal.
As in Section~\ref{ssec:induced}, we assume 
that the partition $\{ \dlj \}$ (equivalently $\{ X_i \}$) 
is generating and that all returns to $\Delta_0$ satisfy $f_\Delta^\tau(\dlj) = \Delta_0$.\footnote{
The abstract setup in \cite{BDM} uses the more general {\em finite images} condition, but
since the towers constructed in \cite{BDM} actually satisfy full returns to a single base,
we will use this simpler version here.}  We assume that the tower has
exponential returns:
\begin{itemize}
  \item[(P1)]  There exist constants $C, \alpha >0$
such that  $\bm(\Delta_n) \le Ce^{-\alpha n}$, for $n \in \N$ 
(this is equivalent to saying that $m(\tau=n)=O(e^{-\alpha n})$).  
\end{itemize}

The tower inherits a natural metric adapted to the dynamics as follows.  
Define the \emph{separation time} on $\Delta$ to be
\[
\begin{split}
s(x,y) = \min \{ n \geq 0 : \; & f_\Delta^nx, f_\Delta^ny \mbox{ lie in different partition 
elements $\dlj$} \}. \\
\end{split}
\]
$s(x,y)$ is finite $\mu$-almost everywhere for any $\mu$ that lifts to the tower
since $\{ X_i \}$ is a generating partition for $f^\tau$.
Choose $\beta \in (0, \alpha)$ and define a metric $d_\beta$ on $\Delta$ by
$d_\beta(x,y) = e^{-\beta s(x,y)}$.  

We introduce a hole $H$ in $\Delta$ which is the union of countably many
partition elements $\dlj$, \ie $H = \cup_{\ell,k} H_{\ell,k}$ where $H_{\ell,k} = \dlj$
for some $j$.  Set $H_\ell = \cup_j H_{\ell,j} \subset \Delta_\ell$.
For simplicity we assume that the base $\Delta_0$ contains no holes (this can always
be arranged in the construction of the tower by choosing a suitable reference set $X$).
We assume the following additional properties for the tower.

\begin{itemize}
  \item[(P2)]  (Bounded Distortion) We suppose that $\phi_\Delta$ is Lipshitz in the metric $d_\beta$.
    Furthermore, we assume there exists $C_d>0$ such that for all $x, y \in \Delta$ and $n \geq 0$,
\begin{equation}
\label{eq:bounded dist delta}
\Big| e^{S_n\phi_\Delta(x) - S_n\phi_\Delta(y)} -1 \Big| \le C_d d_\beta(f_\Delta^nx, f_\Delta^ny) .
\end{equation}
  \item[(P3)] (Subexponential growth of potential)
 For each $\ve >0$, there exists $C>0$, such that 
 \begin{equation}
\label{eq:finite pressure}
 |S_\tau\phi_\Delta | \le  C e^{\ve \tau} \; \;\; \mbox{for all return times $\tau$.}
\end{equation} 

 \item[(P4)]  (Smallness of the hole) 
Let $H_\ell = \cup_j H_{\ell, j}$ and set
$q := \sum_{\ell \ge 1} \bm(H_\ell) e^{\beta (\ell-1)}$.  We assume
\begin{equation}
\label{eq:hole size}
q < \frac{(1-e^{-\beta})\bm(\Delta_0)}{1+C_d} .
\end{equation}
\end{itemize}
We say the open system is mixing if g.c.d.$\{\tau|_{\Delta_0 \cap \f_\Delta^{-\tau} (\Delta_0)} \}=1$
and $\f_\Delta$ still admits at least one return to $\Delta_0$. 

Following Section~\ref{ssec:transfer}, we define the transfer operator 
$\Lp_{\phi^H_\Delta}$ associated with the punctured potential $\phi_\Delta^H$ 
and acting on $L^1(\bm)$ by
\[
\Lp^n_{\phi^H_\Delta} \psi(x)  = \sum_{f_\Delta^n y = x} \psi(x) e^{S_n\phi_\Delta(y)} 1_{\hDelta^n}(y) = \Lp^n_{\phi_\Delta}(\psi 1_{\hDelta^n})(x)
\]
where $\hDelta^n = \cap_{i=0}^n f_\Delta^{-i} \hDelta$ as before.
For notational simplicity, we will denote $\Lp_{\phi^H_\Delta}$ simply by
$\hLp$ for the remainder of this section since the potential $\phi_\Delta$ is fixed.  When we
wish to vary the potential in later sections, we will reintroduce subscripts to 
reinforce the explicit dependence on the potential and the hole.


\subsection{A Spectral Gap for the Transfer Operator}

We define the following function spaces on $\Delta$ used in \cite{young tower, BDM}
on which the transfer operator $\hLp$ for the tower with a hole has a spectral gap.   

For $\psi \in L^1(\bm)$,
define
\begin{equation}
\label{eq:norms}
\begin{split}
\| \psi_{\ell,j} \|_{\lip} &= e^{-\beta\ell} \text{Lip}(\psi|_{\Delta_{\ell,j}}) , \\
\| \psi_{\ell,j} \|_\infty &= e^{-\beta\ell} \sup_{\Delta_{\ell,j}} |\psi| \, .
\end{split}
\end{equation}
Then define $\| \psi \|_{\lip} = \sup_{\ell,j} \| \psi_{\ell,j} \|_{\lip}$,
$\| \psi \|_{\infty} = \sup_{\ell,j} \| \psi_{\ell,j} \|_{\infty}$,
and $\| \psi \|_\B = \| \psi \|_{\lip} + \| \psi \|_\infty$.

Note that if $\| \psi \|_\infty < \infty$, then
\[
\int_\Delta \psi \, d\bm \leq \| \psi \|_\infty \sum_\ell e^{\beta \ell} \bm(\Delta_\ell)
\le \| \psi \|_\infty \sum_\ell e^{\beta \ell} m(\tau \geq \ell) < \infty
\]
by choice of $\beta$ so that $\psi \in L^1(\bm)$.
Let $\B=\B_\beta := \{ \psi \in L^1(\bm) : \| \psi \|_\B < \infty \}$. 

The following theorem is proved in \cite{BDM}.

\begin{theorem}(\cite{BDM})
\label{thm:spectrum}
Let $(f_\Delta, \Delta, H)$ satisfy (P1)-(P4) and assume the open system is mixing.  
Then $\hLp$ has essential spectral radius bounded by $e^{-\beta}$ and
spectral radius given by $e^{-\beta} < \lambda<1$.  The eigenvalue $\lambda$ is simple
and all other eigenvalues have modulus strictly less than 1.

Let $g \in \B$ denote the unique probability density corresponding to
$\lambda$.  Then there exists $\sigma <1$ and $C>0$
such that for all $\psi \in \B$,
\begin{enumerate}
  \item[(i)] $\| \lambda^{-n} \hLp^n \psi - c(\psi) g \|_\B \le C \| \psi \|_\B \sigma^n $, 
  where $c(\psi)$ is a 
  constant depending on $\psi$;
  \item[(ii)]  moreover, $c(\psi) > 0$ if and only if
  $\lim_{n \to \infty} \| \hLp^n \psi /|\hLp^n \psi|_1 - g \|_\B = 0$, where convergence is at the rate
  $\sigma^n$.
\end{enumerate}
\end{theorem}

The primary object of interest in open systems is the limit in (ii) above since it
describes the class of densities whose escape rate matches that of $g$.  Thus it
raises the question, for which functions is $c(\psi)>0$?  We characterize this set more
precisely in the next section by constructing an invariant measure on the survivor set,
which will also serve as an equilibrium state for the open system.


\subsection{Construction of equilibrium state}
\label{ssec:equilibrium}

We begin with what is by now a standard construction of an invariant measure
on the survivor set $\hDelta^\infty := \cap_{n=0}^\infty \f_\Delta^{-n} \hDelta$.
Let $\B_0$ denote the space of functions $\B_\beta$, but with $\beta =0$.

Note that by the conditional invariance equation $\hLp g = \lambda g$, and the fact
that $g \geq \ve >0$ on $\Delta_0$ (\cite[Proposition 2.4]{BDM}), 
we have $C \lambda^{-\ell} \le g|_{\hDelta_\ell} \le C^{-1} \lambda^{-\ell}$,
for some constant $C>0$ depending only on (P1)-(P4).
Since $\lambda > e^{-\beta}$ by construction, it follows from the definitions of
$\B$ and $\B_0$ that $g \psi \in \B$ whenever $\psi \in \B_0$.

Take $\psi \in \B_0$.  Since $g \psi \in \B$, by Theorem~\ref{thm:spectrum}(i), we may 
define,
\[
\mathcal{Q}(\psi) := \lim_{n \to \infty} \lambda^{-n} g^{-1} \hLp^n(g \psi) = c(g \psi) .
\]
This defines a linear functional on $\B_0$.  We also have
$|\hLp^n(g \psi)| \leq |\psi|_\infty \hLp^n g = |\psi|_\infty \lambda^n g$, so that
$|\mathcal{Q}(\psi)| \leq |\psi|_\infty$ and $\mathcal{Q}$ is a bounded, linear functional
on $\B_0$.
By the Riesz representation theorem, there exists a measure $\tnu$ such that
$\mathcal{Q}(\psi) = \tnu(\psi)$ for all $\psi \in \B_0$.  Since $\mathcal{Q}(1)=1$, it follows
that $\tnu$ is a probability measure.  Indeed, it is easy to check that
$\tnu$ is an invariant probability measure for $f_\Delta$ supported on $\hDelta^\infty$.

It follows from \cite[Sect.~3.3]{BDM} that $\tnu$ is ergodic and enjoys exponential decay of
correlations on functions in $\B_0$.

\cite[Section 3.4]{BDM} uses $\tnu$ to formulate the following proposition regarding convergence and
escape rates.

\begin{proposition} (\cite{BDM})
\label{prop:rates}
Let $(f_\Delta, \Delta, \bm, H)$ be as in the statement of Theorem~\ref{thm:spectrum}.
\begin{enumerate}
  \item[(i)]  For each $\psi \in \B_0$ with $\psi \ge 0$, we have $\tnu(\psi) > 0$ if and only if
   \[
   \lim_{n \to \infty} \frac{\hLp^n \psi}{|\hLp^n \psi|_1} = g, 
   \]
where, as usual, the convergence
is at an exponential rate in the $\| \cdot \|$-norm.  
In particular, the reference measure converges to the
conditionally invariant measure $g d\bm$.
  \item[(ii)]  Let $\psi \in \B$, $\psi \geq 0$, with $\tnu (x : \psi(x) > 0) > 0$.  Then the
  limit in (i) holds.  Moreover,  
  the escape rate with respect to the measure $\psi \bm$
  exists and equals $-\log \lambda$.
\end{enumerate}
\end{proposition}


\subsection{Variational Principle on $\hDelta$}
\label{sec:var delta}

Let $\tnu_0:= \frac{1}{\tnu(\Delta_0)} \tnu|_{\Delta_0}$ and define
$\hF=\f_\Delta^\tau: \hDelta^\infty\cap\Delta_0 \circlearrowleft$.
Recall that $\tau^n(x) = \sum_{k=0}^{n-1} \tau(\hF^kx)$ denotes the $n^{\mbox{\scriptsize th}}$
return time
starting at $x$ and let $\M_{\hF}$ be the set of $\hF$-invariant Borel probability
measures on $\hDelta^\infty\cap\Delta_0$.
We will need the following two lemmas, the first of which is proved in \cite[Lemma 5.3]{BDM}.

\begin{lemma} 
\label{lem:nu bounds}
Let $X_{n,i} \subset \Delta_0$
denote a cylinder set of length $n$ with respect to $\hF$.  Then there exists
a constant $C>0$ such that for all $n$,
\[
C^{-1} \lambda^{-\tau^n(y_*)} e^{S_{\tau^n}\phi_\Delta(y_*)} \leq \tnu(X_{n,i})
\leq C \lambda^{-\tau^n(y_*)} e^{S_{\tau^n}\phi_\Delta(y_*)}
\]
where $y_*$ is an arbitrary point in  $X_{n,i}$.
\end{lemma}

The following lemma is missing from \cite{BDM} and is added here as a correction 
in the abstract setting
(see also \cite{DWY2} for a similar correction).

\begin{lemma}
\label{lem:nu finite}
The measure $\tnu_0$ satisfies $\tnu_0(-S_\tau \phi_\Delta) < \infty$ and $\tnu(\tau) < \infty$.
\end{lemma}

\begin{proof}
That $\tnu_0(\tau) < \infty$ is trivial since $\tnu_0$ is a restriction of $\tnu$ and $\tnu(\Delta)=1$:
\[
\int_{\Delta_0} \tau \, d\tnu_0 = (\tnu(\Delta_0))^{-1} \sum_n n \tnu(\tau = n) 
= (\tnu(\Delta_0))^{-1} \tnu(\Delta) < \infty .
\]

To show that $\tnu_0(-S_\tau \phi_\Delta) < \infty$, we use the bounds given by
Lemma~\ref{lem:nu bounds} as well as assumption (P3).  
Note that by definition of conformal measure,
we have
\begin{equation}
\label{eq:pot bound}
e^{S_\tau \phi_\Delta(y_i)} = \frac{ \bm(X_i)}{\bm(\Delta_0)} 
\end{equation}
for some $y_i \in X_i$ and each $i$.
Choosing $\ve <  \alpha - \beta$ in \eqref{eq:finite pressure} and setting
$c_0 = (\tnu(\Delta_0))^{-1}$, we write
\[
\begin{split}
\int_{\Delta_0} - S_\tau \phi_\Delta \, d \tnu_0 &  
\leq c_0 \sum_i \big|S_\tau \phi_\Delta|_{X_i}\big|_\infty  \tnu(X_i) 
\leq C \sum_i e^{\ve \tau(X_i)} \lambda^{-\tau(X_i)} e^{S_\tau \phi_\Delta(y_i)} 
\\
& \leq C' \sum_n  e^{\ve n} \lambda^{-n} \bm(\tau = n) 
\le C'' \sum_n e^{-(\alpha +\log \lambda  - \ve)n} .
\end{split}
\]
Recall that $\lambda > e^{-\beta}$ so that $\log \lambda  > - \beta$.  Thus the exponent
in the sum above is greater than $\alpha - \beta - \ve > 0$
by choice of $\ve$, and so the series converges.
\end{proof}

Lemmas~\ref{lem:nu bounds} and \ref{lem:nu finite} imply that $\tnu_0$ is a Gibbs measure 
with respect to the
potential $S_\tau \phi_\Delta - \tau \log \lambda$.

Notice that for $x \in \Delta_0$,
$S_\tau \phi_\Delta (x) = \sum_{i=0}^{\tau(x)-1} \phi_\Delta (f_\Delta^ix)$.  However,
$\phi_\Delta(f_\Delta^ix) =0$ for $i<\tau(x)-1$, so that 
$S_\tau \phi_\Delta(x) = \phi_\Delta(f_\Delta^{\tau-1}x)$.
Using this, for $\eta_0 \in \M_{\hF}$, we have
\begin{equation}
\label{eq:first int}
\int_{\Delta_0} S_\tau \phi_\Delta \, d\eta_0 
= \eta(\Delta_0)^{-1} \int_{f_\Delta^{-1}\Delta_0} \phi_\Delta \,d\eta
= \eta(\Delta_0)^{-1} \int_\Delta \phi_\Delta \,d\eta.
\end{equation}
so that $\eta(-\phi_\Delta) < \infty$ if and only if 
$\eta_0(-S_\tau \phi_\Delta) < \infty$.  Thus there is a 1-1 correspondence
between the relevant measures in $\M_{\hF}$
and $\M^H_{f_\Delta}$, the set of ergodic, $f_\Delta$ invariant probability measures
supported on $\hDelta^\infty$.   
This implies in particular that $\tnu(-\phi_\Delta) < \infty$
by Lemma~\ref{lem:nu finite} so that $\tnu \in \M^H_{f_\Delta}$.
This leads to the following
equilibrium principle for $f_\Delta$.

\begin{proposition} (\cite[Theorem 2.9]{BDM})
\label{prop:var tower}
Suppose $\psi \in \B$, $\psi \geq 0$, satisfies $\tnu(\psi) > 0$ and $\int \psi \, d\bm =1$.
Let $\e(\bm_\psi)$ be the escape rate of $\bm_\psi := \psi \bm$ from $\hDelta$.  Then
\[
- \e(\bm_\psi) = \log \lambda = \sup_{\eta \in \M^H_{f_\Delta}} \left\{  h_\eta(f_\Delta)
  + \int_\Delta  \phi_\Delta \, d\eta  : - \int_{\Delta} \phi_\Delta \, d\eta < \infty \right\} .
\]
Moreover, $\tnu$ is the only nonsingular $\f_\Delta$-invariant probability measure which
attains the supremum. 
\end{proposition}


\section{Proof of Theorem~\ref{thm:accim}}

\label{sec:project}

In this section, we return to our specific class of maps described in Sect.~\ref{sec:setup}
and use the results of Section~\ref{sec:LY} to obtain conditionally
invariant measures absolutely continuous with respect to the 
$(\vf_t - p_t)$-conformal measures $m_t$, where
$\vf_t = - t \log |Df|$
and $p_t = P_{\M_f}(\vf_t)$.  In order to invoke the results of Section~\ref{sec:LY},
we first verify properties (P1)-(P4) of the constructed towers.

To distinguish between holes in $I$ and $\Delta$, we shall denote by $H$ the hole in
$I$ and by $\tH = \pi^{-1}H$ the hole in $\Delta$.  Thus for consistency,
$\tH = \cup_{\ell \ge 1} \tH_\ell$ and $\tH_\ell = \cup_j \tH_{\ell,j}$.


\subsection{(P1)-(P4) are satisfied with uniform constants}
\label{ssec:uniform p1-p4}

We fix $\Crit_{\mbox{\tiny hole}}$ and $\delta_0>0$ as in Section~\ref{sec:tails}.
Then for $H \in \H(h)$ with $h$ sufficiently small,
by \cite{BDM} we have an inducing scheme and
Young tower satisfying properties
(A1) and (A2).  Let $\bm_t$ denote the reference measure on $\Delta$ induced by $m_t$, 
the $\phi_t$-conformal measure.  
Recall that by Theorem~\ref{thm:lift}, this measure is guaranteed to lift to $\Delta$ if $t\in [t_0, t_1]$. 

By Proposition~\ref{prop:uniform tails}, we have (P1) satisfied uniformly with
respect to $\bm_t$ 
for some $\alpha>0$ (the same $\alpha$ as in Proposition~\ref{prop:uniform tails})
and all $t \in [t_0, t_1]$.  
We choose $\beta \in (0, \alpha)$ and add the restriction that 
$\beta \le t_0 \log \xi$ (see the proof of Lemma~\ref{lem:A2}).  
Then (P2) follows from Lemma~\ref{lem:distortion} and (A1)(a) with a possibly
larger constant $C_d$
for the potentials $\phi_{\Delta,t}$ induced by $\phi_t = \vf_t - p_t$. 
(P3) is automatic
for our class of maps since $|Df|$ is bounded above and due to (A1)(a), we have
$|S_\tau(\phi_t)| \le C \tau$ at return times.  Again, all constants are uniform
for $t \in [t_0, t_1]$.

It remains to verify (P4) for the constructed towers.
 We do this via the following lemma.

\begin{lemma}
\label{lem:uniform P4} 
There exists $h>0$ sufficiently small such that if $H\in\H(h)$ then (P4) is satisfied with respect to the measure
$\bm_t$ for all $t \in [t_0, t_1]$.
\end{lemma}

\begin{proof}
We need to show,
\beq
\label{eq:qt}
\sum_{\ell \ge 1} \frac{\bm_t(\tH_\ell)}{\bm_t(\Delta_0)} e^{\beta (\ell -1)} 
< \frac{1- e^{-\beta}}{1 + C_d}.
\eeq

First assume that $t\in [t_0,1]$.
Recall that each component $\tH_{\ell,j} \subset \Delta$ is a 1-cylinder for the tower
map $f_\Delta$.  
We have for some $y \in \tH_{\ell,j}$,
\beq
\label{eq:equiv t}
\begin{split}
\frac{\bm_t(\tH_{\ell,j})}{\bm_t(\Delta_0)} 
& = e^{S_\tau \phi_{\Delta,t}(y)} = e^{tS_\tau \vf_\Delta(y) - \tau(y) p_t} \\
& = e^{(t-1)S_\tau \vf_\Delta(y) - \tau(y)p_t} e^{S_\tau \vf_\Delta(y)} \\
& \le C_d e^{\tau(f^{-\ell}\tH_{\ell,j}) [(1-t) \chi_M - p_t ]} \frac{\bm_1(\tH_{\ell,j})}{\bm_1(\Delta_0)},
\end{split}
\eeq
where $\chi_M$ is as in \eqref{eq:chi}.  
We use this to estimate \eqref{eq:qt}, 
\[
\sum_{\ell \ge 1} \frac{\bm_t(\tH_\ell)}{\bm_t(\Delta_0)} e^{\beta (\ell -1)}  \le C_d \delta_1^{-1} \sum_{\ell \ge 1} \sum_{j} e^{\beta (\ell -1)} 
e^{\tau(f^{-\ell}\tH_{\ell,j})[(1-t) \chi_M - p_t]} \, \bm_1(\tH_{\ell,j}) ,
\]
where $\delta_1 = m_1(X)$.

Note that $f_\Delta^{-\ell}\tH_{\ell,j}$
is a 1-cylinder for the induced map $F:X \to X$. 
Set 
\[
b_t  = (1-t)\chi_M - p_t  \; \; \mbox{and} \; \;
A_n = \{ \tH_{\ell,j} : \tau(f_\Delta^{-\ell}\tH_{\ell,j}) = n \} .
\]
Then since $\ell \le \tau(f^{-\ell} (\tH_{\ell,j}))$, our estimate becomes,
\[
\sum_{\ell \ge 1} \frac{\bm_t(\tH_\ell)}{\bm_t(\Delta_0)} e^{\beta (\ell -1)} 
 \le C \sum_{n \ge 1} \sum_{\tH_{\ell,j} \in A_n} e^{\beta \ell} e^{n b_t} \bm_1(\tH_{\ell,j})
 \le C \sum_{n \ge 1} e^{(\beta + b_t - \alpha_1) n},
 \]
 since $A_n \subset \{ \tau = n \}$.  Note that $\alpha_1 - b_t \ge \alpha$ and $\beta < \alpha$ 
 by choice of
 $[t_0,t_1]$ and $\beta$ so that the sum is uniformly bounded for $t$ in this interval.

In order to show that the sum can in fact be made arbitrarily small, we split it into two parts,
depending on whether $\tH_{\ell,j}$ is created by an intersection of $f^\ell(X)$ with $H$
during a free period or during a bound period.  Thus
\[
\sum_{\ell,j} e^{\beta \ell} e^{\tau_{\ell,j} b_t} \bm_1(\tH_{\ell,j})
= \sum_{\mbox{\scriptsize bound}} e^{\beta \ell} e^{\tau_{\ell,j} b_t} \bm_1(\tH_{\ell,j}) 
+ \sum_{\mbox{\scriptsize free}} e^{\beta \ell} e^{\tau_{\ell,j} b_t} \bm_1(\tH_{\ell,j}) ,
\]
where $\tau_{\ell,j} = \tau(f_\Delta^{-\ell} \tH_{\ell,j})$.

To estimate the sum over bound pieces, we use the slow approach condition
(H1).  Suppose $\omega \subset X$ is a 1-cylinder in $X$ such that
$f^n(\omega) \subset H$ during a bound period, $n < \tau(\omega)$, and $c \in \Crit_c$ is 
the last critical point visited by $\omega$ at time $n-\ell$.  Since $\omega$ is bound,
we have $|f^\ell x - f^\ell c| \le \delta_0 e^{-2\vartheta_c \ell}$ for all $x \in f^{n-\ell}\omega$
by \cite[Sect. 2.2]{DHL}.  This implies that $\dist(f^\ell c, \partial H) \le m_1(H) + \delta_0e^{-2\vartheta_c \ell}$.
On the other hand, (H1) requires $\dist(f^\ell c, \partial H) \ge \delta_0 e^{- \vartheta_c \ell }$.
This forces, 
\[
\delta_0 e^{- \vartheta_c \ell} \le m_1(H) + \delta_0e^{-2\vartheta_c \ell}
\implies \ell \ge - \vartheta_c^{-1} \log(h/\delta_c) ,
\]
where $\delta_c = \delta_0 (1-e^{-\vartheta_c})$.
Thus since $\ell \le \tau_{\ell,j}$,
\[
\begin{split}
\sum_{\mbox{\scriptsize bound}} e^{\beta \ell} e^{\tau_{\ell,j} b_t} \bm_1(\tH_{\ell,j}) 
& \le \sum_{\mbox{\scriptsize bound}} e^{(\beta + b_t) \tau_{\ell,j}} \bm_1(\tH_{\ell,j}) \\
& \le \sum_{n > - \vartheta_c^{-1} \log (h/\delta_c)} C e^{(\beta+b_t - \alpha_1) n}
\le C' h^{\vartheta_c^{-1}(\alpha - \beta)}.
\end{split}
\]

To estimate the sum over free pieces, we use the following estimate from 
\cite[Proof of Lemma 4.5]{BDM},
\[
\sum_{\mbox{\scriptsize free}} \bm_1(\tH_{\ell,j}) \le C m_1(H) .
\]
Then
\[
\begin{split}
\sum_{\mbox{\scriptsize free}} e^{\beta \ell} e^{\tau_{\ell,j} b_t} \bm_1(\tH_{\ell,j}) 
& \le \sum_{\tau_{\ell,j} \le - \log h} e^{(\beta + b_t)\tau_{\ell,j}} \bm_1(\tH_{\ell,j}) 
+ \sum_{\tau_{\ell,j} > - \log h} e^{(\beta + b_t) \tau_{\ell,j}} \bm_1(\tH_{\ell,j})  \\
& \le e^{- (\beta + b_t) \log h}  \sum_{\tau_{\ell,j} \le - \log h} \bm_1(\tH_{\ell,j})
+ \sum_{\tau_{\ell,j} > - \log h} e^{(\beta + b_t - \alpha_1) \tau_{\ell,j}}  \\
& \le C h^{1-(\beta + b_t)} + h^{\alpha_1 - \beta - b_t},
\end{split}
\]
where all exponents are positive due to the choice of $\beta$ and $t_0$.

The argument for $t \in [1, t_1]$ is similar with $b_t$ defined by
$(1-t) \chi_m - p_t$.
\end{proof}


\subsection{Pushing forward densities on $I$}
\label{ssec:convergence}

We have proved that (P1)-(P4) hold with uniform constants for all $t \in [t_0, t_1]$
and all $H \in \H(h)$ for $h$ sufficiently small.  We now fix such an $H \in \H(h)$.
By Theorem~\ref{thm:spectrum}, for each $t$, we have a conditionally invariant density 
$\tg^H_t \in \B$ satisfying 
$\hLp_{\phi_{\Delta, t}} \tg^H_t = \lambda^H_t \tg^H_t$, where $\lambda^H_t<1$ is a
simple eigenvalue of $\hLp_{\phi_{\Delta,t}}$ with maximum modulus. 
We use the spectral gap for $\hLp_{\phi_{\Delta,t}}$ on $\hDelta$ to obtain information
about the evolution of densities under the action of $\hLp_{\phi_t}$ on $I$.

The philosophy is the following.  For $\tilde{\psi} \in \B$, let 
\[
\pa_{\pi,t} \tilde\psi(x) = \sum_{y \in \pi^{-1}x}
\frac{\tilde{\psi}(y)}{ J_t \pi(y)}
\]
where $J_t \pi$ is the Jacobian of $\pi$ with respect to the measures
$m_t$ and $\bm_t$.  The commuting relation $f^n \circ \pi = \pi \circ f^n_\Delta$ 
implies
\[
\pa_{\pi,t}( \hLp^n_{\phi_{\Delta, t}} \tilde{\psi} ) 
= \hLp^n_{\phi_t} (\pa_{\pi,t} \tilde \psi)
\]
so that the evolution of densities on $I$ under $\hLp_{\phi_t}$ matches
the evolution of densities on $\Delta$ under $\hLp_{\phi_{\Delta,t}}$ for those
densities in $\pa_{\pi,t} \B$.  
Indeed, $|\pa_{\pi, t} \tilde \psi|_{L^1(m_t)} = | \tilde \psi |_{L^1(\bm_t)}$ 
so that mass is
preserved.

The question of which densities on $I$ can be realized as projections of
elements of $\B$ (or $\B_0$) is addressed in \cite{BDM} and is 
somewhat subtle and system dependent.  
Given $\psi \in \cont^{p}(I)$,
we define $\tilde \psi = \psi \circ \pi$ and it is a consequence of (A1) that 
$\tilde \psi \in \B_0$ for all $p \ge \beta/ \log \xi$, where $\xi>1$ is from (A1)(a) 
\cite[Lemma 4.1]{BDM}.
However, in general, $\pa_{\pi,t} (\psi \circ \pi) \neq \psi$ so that this is not sufficient to characterize
those densities which may be realized as projections from of elements in $\B$.

Note that this requirement is different from the problem of liftability of 
the measure $\psi m_t$.  For an invariant measure $\mu$, 
if $\mu$ lifts to an invariant measure
$\tmu$ on $\Delta$, then $\pi_*\tmu = \mu$ as described 
in Section~\ref{ssec:lifting to Delta},
but for a density with respect to $m_t$, this may not be the case
since in general, $\pi_* \bm_t \neq m_t$, even for $t=1$.  
In order to proceed, we will
need the following lemma, which is essentially a version of property (A2)
with respect to the measures $m_t$.

\begin{lemma}
\label{lem:A2}
Let $\mathcal{I} \subset [0,L] \times \N$ be as in (A2).  Then for all $t \in [t_0,t_1]$,
\begin{enumerate}
  \item[(a)] $m_t(I \sm \cup_{(\ell,j) \in \mathcal{I}} \pi(\hDelta_{\ell,j})) =0$
  \item[(b)] $\pi(\hDelta_{\ell_1,j_1}) \cap \pi(\hDelta_{\ell_2,j_2}) = \emptyset$
for all but finitely many $(\ell_1,j_1), (\ell_2,j_2) \in \mathcal{I}$;
  \item[(c)] Define $J_t\pi_{\ell,j} := J_t\pi|_{\hDelta_{\ell,j}}$.  Then
$\sup_{(\ell,j) \in \mathcal{I}}|J_t \pi_{\ell,j}|_\infty 
+ \mbox{Lip}(J_t\pi_{\ell,j}) < \infty$.
\end{enumerate}
As a consequence, $\cont^p(I) \subset \pa_{\pi,t}( \B_0)$ for all 
$p \ge \beta / \log \xi$,
where $\xi >1$ is from (A1).
\end{lemma}

\begin{proof}
To prove (a), recall that if we ignore cuts due to the countable 
exponential partition of $B_{\delta_0}(c)$ for each $c \in \Crit_c$, then
$\pi(\Delta_\ell)$ consists of finitely many intervals.  Thus according
to the proof of (A2)(a) in \cite[Lemma 4.6]{BDM},
$I \sm \cup_{(\ell,j) \in \mathcal{I}} \pi(\hDelta_{\ell,j})$ contains at most the endpoints
of these finitely many intervals together with the images of the cuts of the
exponential partition.  This set is countable and so its $m_t$ measure is zero.

Item (b) is independent of the measure and so is trivially true by (A2)(b).

It remains to prove (c). For $x \in \Delta_\ell$, 
let $x_{-\ell} = f_\Delta^{-\ell}x \in \Delta_0$.  Then by conformality and
the definition of $\bm_t$, we have
\begin{equation}
\label{eq:J_1}
J_t\pi(x) = \frac{dm_t(\pi x)}{d\bm_t(x)} = e^{-S_\ell \phi_t(x_{-\ell})} 
= e^{-t S_\ell \vf(x_{-\ell}) + \ell p_t}
= (J_1 \pi(x))^t e^{\ell p_t} .
\end{equation}
Since $\ell \le L$ by definition of $\mathcal{I}$ and due to property (A2)(c)
of $J_1\pi$, the above relation implies the required bound on the 
$L^\infty$-norm of $J_t\pi$ restricted to
elements of $\mathcal{I}$.

To prove the bound on the Lipschitz constant of $J_t\pi$, we restrict our attention
to the case $t \in [t_0,1)$ since for $t \ge 1$, the Lipschitz property of
$J_t\pi$ follows from that of $J_1\pi$.  Now using the fact that
$|a^t - b^t| \le |a-b|^t$ for $t<1$, we use \eqref{eq:J_1} to estimate for 
$x, y \in \Delta_{\ell,j}$,
\[
|J_t\pi(x) - J_t\pi(y)| \le |J_1\pi(x) - J_1\pi(y)|^t e^{\ell p_t}
\le |Df^\ell(\pi(x_{-\ell})) - Df^\ell(\pi(y_{-\ell}))|^t e^{\ell p_t} .
\]
Since $\ell \le L$ and $f$ is $\cont^2$, this bound yields,
\begin{equation}
\label{eq:jt holder}
|J_t\pi(x) - J_t\pi(y)| \le C |\pi(x_{-\ell}) - \pi(y_{-\ell})|^t .
\end{equation}
Let $s_0 = s(x_{-\ell},y_{-\ell})$.  Since $s_0$ is a return time for 
$x_{-\ell}, y_{-\ell}$, we have $|Df^{s_0}| \ge (C_d')^{-1} \xi^{s_0}$ by
(A1)(a). Thus
\[
|\pi(x_{-\ell}) - \pi(y_{-\ell})| \le C_d' \xi^{-s_0} |f^{s_0}(\pi x_{-\ell}) -
f^{s_0} (\pi y_{-\ell})| \le C_d' \xi^{-s_0} \mbox{diam}(X).
\]
Putting this together with \eqref{eq:jt holder} yields
\[
|J_t\pi(x) - J_t\pi(y)| \le C \xi^{-t s_0} \le C e^{-\beta s_0} \le C d_\beta(x,y)
\] 
since $s(x,y) = s_0 - \ell$ and as long as $\xi^{-t} \le e^{-\beta}$, which is true
for $t \ge t_0$ by choice of $\beta \le t_0 \log \xi$.  This completes the proof of (c).

Now using properties (a)-(c) for $J_t\pi$, it follows from
\cite[Proposition~4.2]{BDM} that
$\cont^{p}(I) \subset \pa_{\pi,t} (\B_0)$ for all $p \ge \beta / \log \xi$.
\end{proof}

For $p \ge \beta/ \log \xi$, define $\mathcal{D}^p(I)$ to be the set of nonnegative functions 
$\psi \in \cont^p(I)$ with  $\psi >0$ on $X$.
The following proposition completes the proof of Theorem~\ref{thm:accim}.

\begin{proposition}
\label{prop:conv}
Let $\tmu^H_t = \tg^H_t \bm_t$ and define $\pi_* \tmu^H_t = \mu^H_t = (\pa_{\pi,t} \tg^H_t) \, m_t$.
Then $\mu^H_t$ is a conditionally invariant measure for $f$ with eigenvalue
$\lambda^H_t$.   In addition,
\begin{enumerate}
  \item[(i)]  For all $\psi \in \mathcal{D}^{p}(I)$, 
\[
\lim_{n \to \infty} \frac{\hLp^n_{\phi_t} \psi}{|\hLp^n_{\phi_t} \psi|_1} = \pa_{\pi,t} \tg^H_t \; \;\; 
\mbox{in } L^1(m_t)
\]
and the convergence occurs at an exponential rate
so that $\mu^H_t$ is a geometric conditionally invariant measure, absolutely
continuous with respect to $m_t$.
  \item[(ii)]  Let $\psi \in \mathcal{D}^p(I)$.  The escape rate with respect to the reference
  measure $\psi m_t$ is given by
  \[   \e(\psi m_t) = - \log \lambda^H_t . \]
 \item[(iii)]  Let $\mu_t$ be the equilibrium state for the potential $\phi_t = \vf_t - p_t$ before the
 introduction of the hole.  Then for all $\psi \in \mathcal{D}^p(I)$,
 $\e(\psi \mu_t) = \e(m_t) = - \log \lambda^H_t$ and
 \[
 \lim_{n \to \infty} \frac{\f^n_* (\psi \mu_t)}{ |\f^n_*(\psi \mu_t)|} = \mu^H_t .
 \]
\end{enumerate}
\end{proposition}

\begin{proof}
The fact that $\pi_* \tmu^H_t$ defines a conditionally invariant measure with the same
eigenvalue as $\tmu^H_t$ follows from the relation $\pi \circ \f_\Delta = \f \circ \pi$.

(i)  Suppose $\psi \in \mathcal{D}^p(I)$.   
By Lemma~\ref{lem:A2}, we may define $\bar \psi \in \B_0$
such that $\pa_{\pi,t} \bar \psi = \psi$.  Then $\tnu_t(\bar \psi) > 0$ 
since $\psi \ge 0$ and $\psi >0$ on $X$ (indeed, this is trivial since
we may always take $X$ to be among the set of elements specified by (A2) to cover $I$). 
Then, by Proposition~\ref{prop:rates}(i),
\begin{equation}
\label{eq:push}
\frac{\hLp^n_{\phi_t} \psi}{|\hLp^n_{\phi_t} \psi|_1} = \frac{\pa_{\pi,t} \hLp^n_{\phi_{\Delta, t}} \bar \psi}
{|\hLp^n_{\phi_{\Delta,t}} \bar \psi|_1} \xrightarrow[]{n \to \infty} \pa_{\pi,t} \tg^H_t ,
\end{equation}
in the $L^1(m_t)$ norm where we have used the fact that
\[
\left| \frac{\pa_{\pi,t} \hLp^n_{\phi_{\Delta,t}} \tilde \psi}{|\hLp^n_{\phi_{\Delta,t}}|_1} - \pa_{\pi,t} \tg^H_t
\right|_{L^1(m_t)} 
= \left| \pa_{\pi,t} \left( \frac{\hLp^n_{\phi_{\Delta,t}} \tilde \psi}{|\hLp^n_{\phi_{\Delta,t}}|_1} -  \tg^H_t \right)  \right|_{L^1(m_t)} 
=  \left|  \frac{\hLp^n_{\phi_{\Delta,t}} \tilde \psi}{|\hLp^n_{\phi_{\Delta,t}}|_1} -  \tg^H_t   \right|_{L^1(\bm_t)},
\]
and the convergence is at an exponential rate since the $\| \cdot \|_\B$-norm dominates the
$L^1(\bm_t)$ norm and $\hLp_{\phi_{\Delta,t}}$ has a spectral gap on $\B$.

(ii)  This follows from Proposition~\ref{prop:rates}(iii) since 
\[
\int_{\I^n} \psi \, dm_t = \int_{\I} \hLp^n_{\phi_t} \psi \, dm_t 
= \int_{\hDelta}  \hLp^n_{\phi_{\Delta,t}} \bar \psi \, d\bm_t
= \int_{\hDelta^n} \bar \psi \, d\bm_t .
\]

(iii)  We claim that the measure $\mu_t = g^0_t m_t$ can be realized as the projection of an
element in $\B_0$.  
Consider the tower $\Delta$ before the introduction of the hole.  The arguments
of Section~\ref{sec:LY} hold in the case when $H=\emptyset$ so that $\Lp_{\phi_{\Delta, t}}$
has leading eigenvalue 1 with eigenvector $\tilde{g}^0_t \in \B_0$ which defines
an invariant measure $\tilde \mu_t = \tilde g^0_t \bm_t$.  Then $\pi_* \tilde \mu_t = \mu_t$
and $\pa_{\pi,t} \tilde g^0_t = g^0_t$, proving the claim.  Since $\tilde g^0_t > 0$ on $\Delta_0$, 
we have $\tnu_t(\tilde g^0_t)>0$ so that $g^0_t$ is in the class of densities for which
the relations in (i) and (ii) hold by Proposition~\ref{prop:rates} (although it is discontinuous on $I$).

It follows that $\psi g^0_t$ can also be realized as the projection of the element
$\psi \circ \pi \cdot \tilde{g}^0_t \in \B_0$ for any $\psi \in \cont^p(I)$.  If in addition,
$\psi \in \mathcal{D}^p(I)$, then $\tnu_t(\psi \circ \pi \cdot \tilde{g}^0_t)>0$ so that
again, the required limits hold.
\end{proof}

\begin{remark}
\label{remark:conv}
As can be seen from the proof of Proposition~\ref{prop:conv}, the convergence result (i)
and escape rate (ii) hold for any $\psi \in L^1(m_t)$ which can be realized as an element
of $\pa_{\pi,t}(\B_0)$ and satisfies $\psi >0$ on $X$.  In fact, this second condition can be relaxed
to $\nu^H_t(\psi)>0$ once the equilibrium measure $\nu^H_t$
of Theorem~\ref{thm:variational} is introduced.
\end{remark}


\section{Proof of Theorem~\ref{thm:variational}}
\label{sec:var proof}

As verified during the proof of Theorem~\ref{thm:accim}, for $H \in \H(h)$ and
$h$ sufficiently small, 
we have a tower $(f_\Delta, \Delta, \tH)$ respecting the hole, \ie\ such that
$\pi^{-1}H = \tH$ is a union of partition elements $\dlj$, which satisfies
(P1)--(P4) with uniform constants for all $t \in [t_0, t_1]$.

Fix $H \in \H(h)$.
We have an invariant measure
$\tnu_t$ supported on $\hDelta^\infty$ which satisfies the
equilibrium principle of
Proposition~\ref{prop:var tower} and is defined by
$$
\tnu_t (\tilde \psi) = \lim_{n \to \infty} (\lambda^H_t)^{-n} \int_{\hDelta^n} \tilde \psi \tg_t^H \, d\bm_t ,
$$
where $\tg^H_t$ and $\lambda^H_t$ are from Theorem~\ref{thm:spectrum}.

Defining $\nu_t = \pi_* \tnu_t$, we have $\nu_t$ supported on $\I^\infty$ since
$\pi(\hDelta^\infty) \subset \I^\infty$.  Moreover, $\nu_t$ is an invariant measure for
$f$ by the relation $f \circ \pi = \pi \circ f_\Delta$.

For $\psi \in \cont^0(I)$, define $\tilde \psi = \psi \circ \pi$.  Then  
\[
\begin{split}
\nu_t(\psi) & = \tnu_t(\tilde \psi) 
= \lim_{n \to \infty} (\lambda^H_t)^{-n} \int_{\hDelta^n} \tilde \psi \tg^H_t \, d\bm_t 
= \lim_{n \to \infty} (\lambda^H_t)^{-n} \int_{\I^n} 
\pa_{\pi,t}(\psi \circ \pi \cdot \tg^H_t) \, dm_t   \\
& = \lim_{n \to \infty} e^{n\e(m_t)} \int_{\I^n} \psi \,  \pa_{\pi,t}(\tg^H_t) \, dm_t ,
\end{split}
\]
so that $\nu_t$ satisfies the definition of $\nu^H_t$ as defined in the limit given 
in Theorem~\ref{thm:variational}. 

The convergence 
of $\Lp^n_{\phi_t^H} \psi/|\Lp^n_{\phi_t^H} \psi|_{L^1(m_t)}$ (respectively, 
$\Lp^n_{\phi_t^H} (\psi g^0_t)/|\Lp^n_{\phi_t^H} (\psi g^0_t) |_{L^1(m_t)}$) 
to $g^H_t := \pa_{\pi,t}(\tg^H_t)$
for any $\psi \in \cont^p(I)$ with $\nu_t(\psi) >0$ (respectively, $\nu_t(\psi g^0_t) >0$),
follows from Proposition~\ref{prop:rates}, given that both $\psi$ and $\psi g^0_t$ can be realized
as elements of $\pa_{\pi,t}(\B_0)$ as in the proof of Proposition~\ref{prop:conv}.

Finally, we need to show that $\nu_t$ achieves the supremum in the required variational
principle.  We begin by projecting the Variational Principle of 
Proposition~\ref{prop:var tower}
down to $I$.

\begin{lemma}
\label{lem:project jac}
Let $\eta \in \M^H_f$ be such that $\eta$ lifts to $\Delta$ and 
$- \infty < \eta(-\vf) < \infty$.  
Let $\teta \in \M^H_{f_\Delta}$ denote the lift of $\eta$ to $\Delta$.  
Then $\int_I \log Jf \, d\eta = \int_\Delta\log Jf_\Delta \, d\teta$, where
$Jf = |Df|$ and $Jf_\Delta$ is the Jacobian of $f_\Delta$ with respect to 
$\bm_1$.
\end{lemma}

\begin{proof}
Due to the
relation $\pi \circ f_\Delta = f \circ \pi$, we have  for $x \in \Delta$,
\[
J_1\pi(f_\Delta x) Jf_\Delta(x) = Jf(\pi x) J_1\pi(x) .
\]
Thus since $\pi_* \teta = \eta$,
\[
\int_I \log Jf \, d\eta = \int_\Delta \log Jf \circ \pi \, d\teta 
= \int_{\Delta} (\log Jf_\Delta + \log J_1\pi \circ f_\Delta - \log J_1\pi) \, d\teta .
\]
We claim that $\int_{\Delta} (\log J_1 \pi \circ f_\Delta - \log J_1 \pi) \, d\teta = 0$.
Note that if $\int_{\Delta} \log J_1\pi \, d\teta$ were finite, this would be trivial
by the invariance of $\teta$, but we do not assume the finiteness of this integral.

We consider two cases.  If $x \in \Delta_\ell \cap f^{-1}_\Delta(\Delta_0)$, then
$J_1\pi(f_\Delta x) = 1$.  Setting $x_{-\ell} = f_\Delta^{-\ell}x$ as before 
and using \eqref{eq:J_1}, we obtain
\[
\log J_1\pi(f_\Delta x) - \log J_1\pi(x) = S_\ell \vf(\pi (x_{-\ell})).
\]
On the other hand, if $x \in \Delta_\ell \sm f_\Delta^{-1}(\Delta_0)$, then
again by \eqref{eq:J_1},
\[
\log J_1\pi(f_\Delta x) - \log J_1\pi(x) = S_{\ell} \vf (\pi (x_{-\ell})) -
S_{\ell+1} \vf(\pi (x_{-\ell})) = - \vf(\pi x).
\]
Putting these two cases together, we have
\[
\int_{\Delta} (\log J_1 \pi \circ f_\Delta - \log J_1 \pi) \, d\teta
= \int_{f_\Delta^{-1}(\Delta_0)} S_\ell \vf(\pi (x_{-\ell})) \, d\teta 
- \int_{\Delta \sm f_\Delta^{-1}(\Delta_0)} \vf(\pi x) \, d\teta .
\]
Both integrals are finite by assumption on $\teta$.  We decompose $\Delta$ into columns 
$\{ f_\Delta^\ell(X_i) \}_{\ell < \tau(X_i)}$ and note that the first integral
considers $S_\ell \vf(\pi(x_{-\ell}))$ in the top element of the column while the 
second considers the sum of $\vf \circ \pi$ 
in all the levels below the top one, which is precisely
the same thing.  This, plus the fact that $\teta(f_\Delta^\ell(X_i)) = \teta(X_i)$
for $\ell < \tau(X_i)$ provides the required cancellation.
\end{proof}

Taking $\eta \in \M^H_f$ with $-\infty < \eta(-\vf) < \infty$, we use
Lemma~\ref{lem:project jac} to write 
$\eta(\phi_t) = \teta(\phi_{\Delta,t})$ for each $t$
since $\eta(-\phi_t) < \infty$ if and only if $\eta(- \vf)< \infty$. 

Moreover, $h_{\teta}(f_\Delta) = h_\eta(f)$ since $\pi$ is at most countable-to-one
\cite[Proposition 2.8]{buzzi}.  Putting these together yields by 
Proposition~\ref{prop:var tower},
\begin{align*}
-\e(m_t) = \log \lambda^H_t = & \sup_{\eta \in \M^H_f}  \left\{ h_\eta(f) + t \int_I \vf \, d\eta  :
\eta \mbox{ lifts to } \Delta \right\}- P_{\M_f}(t \vf).\\
\end{align*}

However, the condition ``$\eta$ lifts to $\Delta$"  does not suffice to prove 
Theorem~\ref{thm:variational} since that condition is not well understood and depends
on the inducing scheme.  In order to replace the above class of measures with
the class $\G^H_f$ which is independent of the inducing scheme, we must prove
two things: 
\begin{enumerate}
\item[(i)]  $\nu_t \in \G^H_f$; and 
\item[(ii)] $-\e(m_t) \ge P_{\G^H_f}(t\vf) - P_{\M_f} (t\vf)$.
\end{enumerate}
Proving (i) will imply $-\e(m_t) \le  P_{\G^H_f}(t\vf) - P_{\M_f} (t\vf)$ since
we know $\nu_t$ lifts to $\Delta$ and $\nu_t(-\phi_t) < \infty$ by Lemma~\ref{lem:nu finite}
and \eqref{eq:first int}.  Then (ii) will yield the required equality.
We proceed to prove these points in the next two subsections.


\subsection{The weight near the boundary of the hole}

In this section, we prove the following proposition.

\begin{proposition} 
\label{prop:G^H_f}
There exist $C, r>0$ such that
$\nu_t(N_\ve(\partial H \cup \Crit_c)) \le C\ve^r$ for all $\ve>0$,
where $N_\ve(\cdot)$ denotes the $\epsilon$-neighborhood of a set.
\end{proposition}

\begin{proof}
Denote by $\Z_n$ the partition of $\Delta_0$ into $n$-cylinders for $F = f^\tau$.
Recall that $\nu_t = \pi_* \tnu_t$ and that $\tnu_0 := (\tnu_t(\Delta_0))^{-1} \tnu_t|_{\Delta_0}$
is a Gibbs measure for $\hF$ by Lemma~\ref{lem:nu bounds} 
which satisfies
\begin{equation}
\label{eq:gibbs}
C^{-1} \lambda_t^{-\tau^n(y_*)} e^{S_{\tau^n}\phi_{\Delta,t}(y_*)} \leq \tnu_0(Z_n)
\leq C \lambda_t^{-\tau^n(y_*)} e^{S_{\tau^n}\phi_{\Delta,t}(y_*)}
\end{equation}
for any $Z_n \in \Z_n$ and $y^* \in Z_n$.  Here $\lambda_t = \lambda^H_t$;
we have retained the explicit dependence
on $t$, but have suppressed dependence on $H$.

Fix $\ve >0$ and choose $n_0 \in \N$ to be the minimal $n$ such that 
\[
\sup_{Z_n \in \Z_n} \sup_{\ell < \tau(Z_n)} m_1(\pi(f_\Delta^\ell Z_n)) < \ve ,
\]
where $m_1$ denotes Lebesgue measure as usual.
Note that such an $n_0$ exists due to the fact that there is exponential expansion
at return times by property (A1)(a).  Indeed, let $Z'_{n-1} = F(Z_n)$ be the $n-1$
cylinder mapped to by $Z_n$.  Then $m_1(\pi (f_\Delta^\ell Z_n)) \le m_1(\pi Z'_{n-1})
\le C_0^{-1} \xi^{-\tau^{n-1}(Z'_{n-1})} m_1(\Delta_0)$, for each $\ell < \tau(Z_n)$.
Thus $m_1(\pi (f_\Delta^\ell Z_n)) \le \ve$ whenever
\[
\tau^{n-1}(Z'_{n-1}) > \frac{- \log (C_0\ve/\delta_1)}{\log \xi} .
\]
Since $\tau_n \geq \tau_{\min} n$, where $\tau_{\min}$ denotes the minimum return time,
it suffices to choose
\begin{equation}
\label{eq:n0}
n_0 > 1 + \frac{- \log (C_0\ve/\delta_1)}{\tau_{\min} \log \xi} .
\end{equation}

For brevity, set $B = \partial H \cup \Crit_c$, the singularity set.
Let $\mathfrak{C}_\ve$ denote the collection of $n$-cylinders $Z_n$ of minimal index $n \le n_0$
such that $\pi(f_\Delta^\ell Z_n) \cap N_\ve(B) \neq \emptyset$ for some $\ell < \tau(Z_n)$.  
By minimal index,
we mean that if $Z_n$ is contained in an $(n-1)$-cylinder $Z_{n-1}$ such that
$\pi(f_\Delta^\ell Z_{n-1}) \subset N_{2\ve}(B)$, then we omit $Z_n$ from
$\mathfrak{C}_\ve$ and include $Z_{n-1}$ instead.
For each $n \le n_0$, define $\mathfrak{C}_{\ve,n}$ to be the set of 
$n$-cylinders in $\mathfrak{C}_\ve$.

Note that 
\[
\nu_t(N_\ve(B)) \le \sum_{Z_n \in \mathfrak{C}_\ve} \sum_{\mbox{\scriptsize relevant $\ell$}}
\tnu_t(f_\Delta^\ell Z_n) .
\]

For $Z_n \in \mathfrak{C}_{\ve}$, there are two possibilities when 
$\pi(f^\ell_\Delta Z_n) \cap N_{\ve}(B) \neq \emptyset$:  either the interval
$\pi(f^\ell_\Delta Z_n)$ is free or it
is bound.  If it is bound at time $\ell$ due to passing through $B_{\delta_0}(\Crit)$ at 
time $\ell - k$, we have
$|f^k x - f^k c| \le \delta_0 e^{-2 \vartheta_c k}$ for all $x \in \pi(f_\Delta^{\ell - k} Z_n)$
and some $c \in \Crit_c$ 
by \cite[Sect. 2.2]{DHL}, where $\vartheta_c$ is from (C2).  This implies that
 dist$(f^k c, B) \le \ve + \delta_0 e^{-2 \vartheta_c k}$.
On the other hand,
the slow approach conditions (C2) and (H1) imply that
dist$(f^k c, B) \geq \delta_0 e^{-\vartheta_c k}$.  Putting these two conditions together,
we must have
\[
\delta_0 e^{-2 \vartheta_c k} - \delta_0 e^{-\vartheta_c k} + \ve \geq 0,
\]
which admits two possibilities:  either
\[
e^{-\vartheta_c k} < \frac{1- \sqrt{1- \frac{4 \ve}{\delta_0}}}{2} \qquad 
\mbox{or} \qquad
e^{-\vartheta_c k} > \frac{1+ \sqrt{1- \frac{4 \ve}{\delta_0}}}{2} .
\]
Since $k \geq 1$ by (H1), we may eliminate the second possibility by only considering
$\ve$ sufficiently small that $e^{-\vartheta_c } < \frac{1+ \sqrt{1- \frac{4 \ve}{\delta_0}}}{2}$.
For the first possibility to occur, we estimate $\sqrt{1-x} \ge 1-x$, for $0 \le x \le 1$,
and solve for $k$ to obtain
the requirement
\begin{equation}
\label{eq:bound large}
k > \frac{-\log(2\ve/\delta_0)}{\vartheta_c} .
\end{equation}
Thus we must have $\tau(Z_n) > \frac{-\log(2\ve/\delta_0)}{\vartheta_c} := s$ if $Z_n$ is to intersect
$N_\ve(B)$ during a bound period.  For cylinders with large return times, we can make a simple estimate using \eqref{eq:pot bound}, \eqref{eq:gibbs} and 
Proposition~\ref{prop:uniform tails},  
\[
\begin{split}
\sum_{\stackrel{Z_n \in \mathfrak{C}_\ve}{\tau(Z_n) > s} } 
\sum_{\mbox{\scriptsize relevant $\ell$}}
\tnu_t(f^\ell_\Delta Z_n)
& \le \sum_{\stackrel{Z_1 \in \Z_1}{\tau(Z_1) > s} } \tau(Z_1) \tnu_t(Z_1)
\le \sum_{\stackrel{Z_1 \in \Z_1}{\tau(Z_1) > s} } \tau(Z_1) \lambda_t^{-\tau}  e^{S_\tau \phi_t(Z_1)} \\
& \le \sum_{\tau > s}  \tau e^{\beta \tau} C_0 e^{-\alpha \tau}
\le C' s e^{(\beta - \alpha)s} \le C'' \ve^{(\alpha - \beta)/\vartheta_c} \log \ve .
\end{split}
\]
where we have used the fact that  $\lambda_t^{-1}  \le e^\beta$ and $\beta < \alpha$.
  
It remains to estimate the contribution from cylinders  with $\tau < s$.  Notice that by
\eqref{eq:bound large}, all of these contributions are from pieces that are free at the time
they intersect $N_\ve(B)$.
We fix $n \leq n_0$ and estimate the contributions from
one $\mathfrak{C}_{\ve,n}$ at a time.  We also fix $t \in [t_0, 1]$.  The argument for
$t \in [1, t_1]$ is similar.

Notice that by definition of $\mathfrak{C}_{\ve,n}$, if $Z_n \in \mathfrak{C}_{\ve,n}$, then 
$\pi(f^\ell_\Delta Z_n) \cap N_\ve(B) \neq 0$ for some (possibly more than one) 
$\ell < \tau(Z_n)$; but $\pi(f^\ell_\Delta Z_{n-1}) \not \subset N_{2 \ve}(B)$ where
$Z_{n-1}$ is the $(n-1)$-cylinder containing $Z_n$.  This implies that
$|\pi(f_\Delta^\ell Z_{n-1})| \geq \ve$.  Then since
$f_\Delta^{\tau^{n-1}(Z_n)}(Z_{n-1}) = \Delta_0$, we have by (A1),
\begin{equation}
\label{eq:small return}
C_0 \xi^{\tau^{n-1}(Z_n) - \ell} |\pi(f_\Delta^\ell Z_{n-1})| \le m_1(\Delta_0)
\implies \tau^{n-1}(Z_n) - \ell \le \frac{- \log(C_0 \ve/\delta_1)}{\log \xi}.
\end{equation}

Since we are restricting to $\tau(Z_n) < s$, there are at most $s$ values of $\ell$
such that $\pi(f_\Delta^\ell Z_n) \cap N_\ve(B) \neq \emptyset$, so
\[
\sum_{\stackrel{Z_n \in \mathfrak{C}_{\ve, n}}{\tau(Z_n) < s}} 
\sum_{\mbox{\scriptsize relevant $\ell$}} \tnu_t(f_\Delta^\ell Z_n) 
\le s \sum_{Z_n \in \mathfrak{C}_{\ve,n}} \tnu_t(Z_n)
\]
and since $s \approx \log \ve$, it suffices to estimate the sum above.

Since $- S_{\tau^n}\vf \le \tau^n \chi_M$, we have,
\begin{equation}
\label{eq:Stau bound}
S_{\tau^n}\phi_{\Delta,t} = t S_{\tau^n} \vf - \tau^n p_t
\le [(1-t)\chi_M - p_t] \tau^n + S_{\tau^n} \vf .
\end{equation}

Setting $c_t = - \log \lambda_t + (1-t) \chi_M - p_t$ for brevity and using \eqref{eq:Stau bound}
together with \eqref{eq:gibbs} and \eqref{eq:small return}, we obtain
\begin{equation}
\label{eq:pull out}
\begin{split}
& \sum_{Z_n \in \mathfrak{C}_{\ve,n}} \tnu_t(Z_n) 
\le \sum_{Z_n \in \mathfrak{C}_{\ve,n}} C \lambda_t^{-\tau^n(Z_n)} e^{S_{\tau^n(Z_n)} \phi_{\Delta,t}}\\
& \le \sum_{Z_n \in \mathfrak{C}_{\ve,n}} C \lambda_t^{-\tau^n(Z_n)} e^{[(1-t)\chi_M - p_t] \tau^n(Z_n)} m_1(Z_n) \\
& \le Ce^{c_t (\tau^{n-1}(Z_n) - \tau(Z_n))} 
\sum_{Z_n \in \mathfrak{C}_{\ve,n}}  e^{c_t (\tau(Z_n) + \tau(F^{n-1}Z_n))} m_1(Z_n) \\
& \le C \ve^{-c_t/\log \xi}
\sum_{Z_n \in \mathfrak{C}_{\ve,n}}  e^{c_t (\tau(Z_n) + \tau(F^{n-1}Z_n))} m_1(Z_n) .
\end{split}
\end{equation}

We split the sum up according to whether
$\tau(Z_n) + \tau(F^{n-1}Z_n)$ is larger or smaller than $-\eta \log \ve$ for some $\eta >0$
to be chosen later.
Note that due to bounded distortion and the tail estimate, we have
$m_1(x \in \Delta_0 : \tau(x) + \tau(F^{n-1}x) = k) \le C e^{- \alpha_1 k}$.  Thus for pieces with
large return times, we have
\begin{equation}
\label{eq:large}
\begin{split}
\sum_{\tau + \tau \circ F^{n-1} > - \eta \log \ve} & e^{c_t (\tau(Z_n) + \tau(F^{n-1}Z_n))} m_1(Z_n)  \\
& \le C (e^{c_t - \alpha_1}) ^{- \eta \log \ve}
\le C \ve^{\eta( \alpha_1 - c_t )} .
\end{split}
\end{equation}
For pieces with small return times, we have
\begin{equation}
\label{eq:small}
\begin{split}
\sum_{\tau + \tau \circ F^{n-1} < - \eta \log \ve} & e^{c_t (\tau(Z_n) + \tau(F^{n-1}Z_n))} m_1(Z_n)  \\
& \le C (e^{c_t}) ^{- \eta \log \ve}
\sum_{\stackrel{Z_n \mbox{ \Small free at time } \ell}{Z_n \in \mathfrak{C}_{\ve,n}}} m_1(f_\Delta^\ell Z_n) \\
& \le C \ve^{- \eta c_t} \ve.
\end{split}
\end{equation}
where we have used the fact that the Lebesgue measure of free pieces that project into an interval of length $\ve$ is bounded by const.$\ve$ (see \cite[Sect. 4.3, Step 1]{BDM}).

In order to prove our estimate, we need the powers of $\ve$ in both \eqref{eq:large}
and \eqref{eq:small} to be positive after multiplying by the factor $\ve^{-c_t/\log \xi}$
appearing in \eqref{eq:pull out}.  Thus we need,
\[
\frac{c_t}{( \alpha_1 - c_t) \log \xi}
< \eta <
\frac{\log \xi - c_t}{c_t \log \xi } .
\]
Such an $\eta$ always exists as long as
\[
c_t = - \log \lambda_t + (1-t) \chi_M - p_t < \frac{\alpha_1 \log \xi}{\alpha_1 + \log \xi},
\]
which holds for all small holes and for all $t$ close to 1 since $p_1=0$ and $\lambda_t \to 1$
as $H$ becomes small.

We have estimated that the contribution to $\nu_t(N_\ve(B))$ from pieces in each $\mathfrak{C}_{\ve,n}$
satisfies the desired bound.  Since there are at most  $n_0$ sets 
$\mathfrak{C}_{\ve,n}$ in $\mathfrak{C}_\ve$,
and $n_0 \le - C \log \ve$ by \eqref{eq:n0}, 
summing over $n$ adds only a logarithmic factor
to our estimate, completing the proof that $\nu_t \in \G^H_f$.
\end{proof}


\subsection{Volume estimate}
\label{sec:volume}

In this section, we will prove the following proposition, which then completes the proof of Theorem~\ref{thm:variational}. 

\begin{proposition}
For each $t \in [t_0, t_1]$,
$$-\e(m_t) \ge \sup_{\mu\in \mathcal{G}_{f}^H} \left\{h(\mu)+ t \int\varphi~d\mu
\right\} 
- P_{\M_f} (t\varphi).$$
\label{prop:lower bound new}
\end{proposition}

We will estimate the $m_t$-mass of $\I^n$ in terms of the pressure using the following partitions.
Define $\P_1$ to be the partition of $I$ into open intervals whose endpoints
are elements of $\Crit_c$ and let 
$\P_n := \bigvee_{k=0}^{n-1}f^{-k}\P_1$.  Similarly, let $\tilde \P_1$ denote the partition of
$I$ induced by $\Crit_c \cup \partial H$ and define $\tilde \P_n$ analogously.
We will estimate the mass of the elements of $\tilde\P_n$ in terms of Lyapunov exponents, and the number these cylinder sets in terms of the entropy.  To get the estimate on the mass we construct another partition using the method of \cite[Section 4]{Dob08}.   Note that this follows a very similar construction given in \cite[Section 2]{Led81}; see also \cite[Theorem 11.2.3]{PrzUrb10}
and \cite[Section 3]{DWY2}.  

We define the natural extension as in  \cite{Led81}.  First define
$$Y:=\{y=(y_0,y_1,\dots):f(y_{i+1}) =y_i\in I\}.$$
Define ${\bar f}^{-1}:Y\to Y$ by ${\bar f}^{-1}((y_0,y_1, \ldots))=(y_1, y_2, \ldots)$, so that ${\bar f}^{-1}$ is invertible with inverse ${\bar f}:{\bar f}^{-1}Y\to Y$ given by ${\bar f}((y_0,y_1, \ldots))=(f(y_0), y_0, y_1, \ldots)$.  The projection $\Pi:Y\to I$ is defined as $\Pi:y=(y_0, y_1,\ldots)\mapsto y_0$.  Hence $\Pi\circ {\bar f}=f\circ \Pi$.  As in 
\cite{Roh61} (see also \cite[Section 2.7]{PrzUrb10}), for any $\mu\in \M$ there is a unique ${\bar f}$-invariant probability measure $\overline\mu$ on $Y$ such that $\Pi_*\overline\mu=\mu$.  Moreover, $\overline\mu$ is an ergodic invariant probability measure for  ${\bar f}^{-1}$.

The triplet $(Y, {\bar f}, \overline\mu)$ is called the \emph{natural extension} of $(I, f,\mu)$.
The following is a mild adaptation of 
 \cite[Theorem 4.1]{Dob08}, see also \cite[Theorem 8]{Led81}.
 
 \begin{theorem}
Suppose that $\mu\in \mathcal{G}^H_f$ has  $\chi:=\int\log|Df|~d\mu>0$ and let $(Y,{\bar f}, \overline\mu)$ denote the natural extension of $(I, f, \mu)$.  Then there exists a measurable function $g$ on $Y$, $0<g<\frac12$ $\overline\mu$-a.e. such that for $\overline\mu$-a.e. $y\in Y$ there exists a set $V_y\subset Y$ with the following properties:
\begin{itemize}
\item $y\in V_y$ and $\Pi V_y=B(\Pi y, g(y))$;
\item for each $n\in \N$, the set $\Pi {\bar f}^{-n}V_y$ is contained in $\tilde\P_n$;
\item for all $y'\in V_y$,
$$\sum_{i=1}^\infty\left|\log|Df(\Pi {\bar f}^{-i} y')|-\log|Df(\Pi {\bar f}^{-i} y)|\right|<\log 2;$$
\item For each $\eta>0$ there exists a measurable function $\rho$ on $Y$ mapping into $[1,\infty)$  a.e. and such that
$$\rho(y)^{-1}e^{n(\chi-\eta)}<|Df^n(\Pi {\bar f}^{-n}y)|< \rho(y)e^{n(\chi-\eta)},$$
in particular, $|\Pi {\bar f}^{-n}V_y|\le 2\rho(y)e^{-n(\chi-\eta)}|\Pi V_y|$.
\end{itemize}

\label{thm:dobbs}
\end{theorem}

The only significant change to the proof given in  \cite[Section 4]{Dob08} is to input information on the rate of approach of typical points to the boundary of the adapted partition $\tilde\P_n$, rather than simply $\P_n$, which is then applied in Lemma 4.5 of that paper.  This information is contained in the following lemma.

\begin{lemma}
Given $\mu\in \mathcal{G}^H_f$, for each $\eta >0$, for $\overline\mu$-a.e. $y\in Y$ there exists $N\in \N$ such that $n\ge N$ implies $d(f^k(\Pi {\bar f}^{-n}y), \bd H)>e^{-\eta n}$ for all $0\le k\le n-1$. 
\label{lem:avoid holes}
\end{lemma}

\begin{proof}
For any subset $A\subset I$ and $\delta>0$, set $\dist_{A,\delta}(x):= d(x,A)$ if $d(x,A)<\delta$ and 1 if 
 $d(x,A)\ge \delta$.  So if $\mu\in \mathcal{G}^H_f$ then $-\int\log \dist_{\bd H,\delta}(x)~d\mu(x)<\infty$.  Moreover, for any $\eps>0$ there exists $\delta>0$ such that $-\int\log \dist_{\bd H,\delta}(x)~d\mu(x)<\eps$.  Since $(Y, {\bar f}^{-1},\overline\mu)$ is an ergodic dynamical system, by the ergodic theorem, for $\overline\mu$-a.e. $y\in Y$,
 \begin{align*}
 \frac1n\sum_{k=0}^{n-1}\log \dist_{\bd H,\delta}(f^k(\Pi {\bar f}^{-n}y))&\to\int\log \dist_{\bd H,\delta}(\Pi y)~d\overline\mu(y)\\
 &=\int\log \dist_{\bd H,\delta}(x)~d\mu(x),\text{ as } n\to \infty. 
 \end{align*}
Now fix $\eta > 0$ and choose $\delta$ so that $- \int \log \dist_{\partial H, \delta}(x) \, d\mu(x) < \eta/2$.  Then by the above limit, for $\bmu$-a.e.\ $y \in Y$, there exists $N = N(y)$ satisfying the
statement of the lemma. 
\end{proof}

We need one more lemma before completing the proof of Proposition~\ref{prop:lower bound new}.

\begin{lemma}
\label{lem:large scale new}
For each $t \in [t_0, t_1]$ and all $\delta>0$, there exists $\delta_t'>0$ such that for every $x\in I$, $m_t((x-\delta, x+\delta))>\delta_t'$. 
\end{lemma}

\begin{proof}
This follows from the fact that $m_t$ gives open sets positive mass and the compactness of $I$.
\end{proof}

\begin{proof}[Proof of Proposition~\ref{prop:lower bound new}]
Fix $t \in [t_0, t_1]$.
For any element $\tilde\cyl_n\in \tilde P_n$,  by construction either $f^{n-1}(\tilde\cyl_n)\subset H$ or $f^{n-1}(\tilde\cyl_n)\cap H=\es$, so either $\tilde\cyl_n$ is contained in $\I^{n-1}$ or it is outside $\I^{n-1}$.  Notice that the partition given by Theorem~\ref{thm:dobbs} is subordinate to $\tilde P_1$.  This will give us subsets of cylinders $\tilde\cyl_n\in \tilde\P_n$ on which we have a good idea of the distortion. 

Fix $\mu \in \G_f^H$ and the corresponding measure $\bmu$ in $Y$.
Set $\eta>0$ and let $\rho$ be as in Theorem~\ref{thm:dobbs}.
For  $\delta,K>0$, let $y\in \overline I_{\delta,K}:=\{y\in Y:|\Pi V_y|>\delta\text{ and } \rho(y)<K\}$.    Fix $\eps>0$ and choose $\delta>0$ small enough and $K$ large enough such that 
$\bmu(\overline I_{\delta,K} ) \ge 1-\eps$.  By invariance,
$\bmu({\bar f}^{-n} (\overline I_{\delta,K})) = \mu(\Pi {\bar f}^{-n}(\overline I_{\delta,K})) \ge 1-\ve$.

Since $\mu$ is supported on $\I^\infty$, for $\bmu$-a.e.\ $y \in \overline I_{\delta,K}$,
this yields $x \in \I^\infty$ such that $x = \Pi {\bar f}^{-n}y $.  Moreover, defining
$$
I_{\delta, K, n}:=\{\Pi {\bar f}^{-n}y:y\in \overline I_{\delta,\kappa} \}\cap \I^{\infty},
$$
we have $\mu(I_{\delta, K, n}) \ge 1-\eps$ for all $n$.  

For every $x\in I_{\delta, K, n}$ we take the corresponding $y\in \overline I_{\delta,K}$ and set $V_{x,n}:=\Pi {\bar f}^{-n}V_y$.  By this setup, $|f^n(V_{x,n})|>\delta$.
By the Mean Value Theorem  and the conformality of $m_t$ there exists $z\in \Pi {\bar f}^{-n}V_y$ such that 
$$
m_t(V_{x,n}) = |Df^n(z)|^{-t}e^{-np_t} m_t(f^n(V_{x,n})).
$$  
Since $V_{x,n} \subset \tilde\cyl_n[x]$ and the two last parts of Theorem~\ref{thm:dobbs} we have
\begin{align}
\label{eq:Cn bound}
m_t(\tilde\cyl_n[x]) & \ge m_t(V_{x,n}) \ge \frac1{2^t} |Df^n(x)|^{-t}e^{-np_t} m_t(f^n(V_{x,n})) \\
& \ge  \frac1{2^tK}  e^{-nt(\chi(\mu)+\eta)} e^{-np_t} m_t(f^n(V_{x,n}))\ge \frac1{2^tK}  e^{-nt(\chi(\mu)+\eta)}e^{-np_t}\delta_t',
\end{align}
for $\delta_t'>0$ depending on $\delta>0$ as in Lemma~\ref{lem:large scale new}.

Now we use the Shannon-McMillan-Breiman Theorem, see for example \cite[Section 2.5]{PrzUrb10}, to assert that on a set $E \subset \I^\infty$ of $\mu$-measure at least $1 - \eps$, 
$\mu(\tilde \cyl_n[x]) \le e^{-n(h(\mu) - \eta)}$ for $x\in E$ and $n$ sufficiently large.
Thus the number of distinct cylinders $\tilde \cyl_n[x]$ with $x \in E \cap I_{\delta, K, n}$ is
at least $(1-2\ve) e^{n(h(\mu) - \eta)}$ for all $n$ large enough.  
Notice that such cylinders are in $\I^{n-1}$
by definition of $\tilde \P_n$.
Combining this with \eqref{eq:Cn bound}, we obtain,
$$
m_t(\I^{n-1})\ge \sum_{x\in I_{\kappa, \delta, n} \cap E} m_t(\tilde \cyl_n[x])\ge   (1-2\ve)
e^{n(h(\mu)-\eta)} \frac{\delta_t'}{2^t K}  e^{-nt(\chi+\eta)}e^{-np_t}.$$
So by the arbitrary choice of $\eta>0$, taking logs of both sides, dividing by $n$ 
and letting $n\to\infty$ yields 
$m_t(\I^{n-1})\ge h(\mu)-t\chi(\mu)$. Taking a supremum over all $\mu\in \mathcal{G}_f^H$, the proposition is proved.
\end{proof}


\section{A Bowen Formula: Proof of Theorem~\ref{thm:hausdorff}}
\label{sec:bowen}

\begin{theorem}
Under the assumptions of Theorem~\ref{thm:variational}, 
$\dim_H(\I^\infty)=t^*$ where $t^*$ is the unique 
value of $t$ such that $P_\Delta(t \varphi^H)=0$.
\end{theorem}

\begin{proof}
We will use the main result of \cite{Iom}, which is a Bowen formula for countable Markov shifts, but similarly to applications in that paper, extends to our case here as follows.
As in Theorem~\ref{thm:lift}(b), for $h$ small enough and $H \in \H(h)$, 
$\dim_H(\I^\infty)=\dim_H(\hDelta^\infty)$.  
The structure of $\Delta$ allows us to code it as a countable Markov shift.
The required result then follows as a consequence of \cite{Iom}, where the 
corresponding `metric potential' as in \cite[(3)]{Iom} is $t \vf_\Delta$
using the H\"older regularity provided by Lemma~\ref{lem:distortion}.

Now \cite[Theorem 3.1]{Iom} gives $\dim_H(\hDelta^\infty)= t^*$ where $t^*$ is the
unique value of $t$ satisfying
$P_{\M^H_{f_\Delta}}(t^* \vf_\Delta) = 0$.  Then following the proof of
Theorem~\ref{thm:variational} from Section~\ref{sec:var proof} yields
$P_\Delta(t^* \vf^H)= 0$.
\end{proof}

Recall that the invariant measure $\nu^H_t$ constructed in Section~\ref{sec:var proof}
achieves the supremum in the variational formula and belongs to $\G^H_f$.
Also, $\nu^H_f$ lifts to $\Delta$, so that it is included in the pressure
$P_\Delta(t \vf^H)$.  Thus $P_{\G^H_f}(t^* \vf^H)= 0$, completing the proof 
of Theorem~\ref{thm:hausdorff}.  


\section{Zero-Hole Limit:  Proofs of Theorems~\ref{thm:zero hole} and \ref{thm:zero hole adapted}}
\label{zero hole}

Recall that for Theorem~\ref{thm:zero hole}, we consider holes of the form
$H_\ve = (z-\ve, z+\ve)$, for $z \in I$ satisfying the following condition: 
\beq
\label{eq:slow approach}
\begin{array}{l}
\mbox{There exist } \varsigma, \delta_z > 0 \mbox{ with } \varsigma < \min \{ 2 \vartheta_c , \alpha/s_t \} , \\
\mbox{such that } |f^n(c) - z| \ge \delta_z e^{-n\varsigma} \mbox{ for all } n \ge 0 .
\end{array}
\eeq
Here $\vartheta_c$ is from (C2),
$\alpha$ is from Proposition~\ref{prop:uniform tails} and $s_t$ is the 
local dimension of $m_t$ at $z$ given in  
Lemma~\ref{lem:scale} below.
The fact that this is a generic condition with respect to both $m_t$ and $\mu_t$ is proved in 
 Lemma~\ref{lem:typical}. Recall that \eqref{eq:slow approach} is part of condition (P).

Because we will need to maintain careful control of the constants involved in the tower
construction along our sequence of holes, we recall the following set of choices explicitly.

Fix $z \in I$ satisfying \eqref{eq:slow approach} so that  $\Crit_{\mbox{\tiny hole}} = \{ z \}$ 
 satisfies (H2) and \eqref{eq:no holes mixing} with $\delta = \delta_0$.  This fixes $n(\delta_0)$ and
 all appearances of $\delta$ in (C1)-(C2) to have value $\delta_0$.
 As before, denote by $\H(h)$ the family of intervals $H$ such that $z \in H$, $m_1(H) \le h$
 and $H$ satisfies (H1).  All our intervals $H_\ve$ are required to belong to $\H(h)$
 for some $h >0$.

\subsection{Preparatory Lemmas and the proof of Theorem~\ref{thm:zero hole}}
\label{prep}

We adopt the following notation for our family of inducing schemes:
$(\cup_i X_\eps^i, F_\ve, \tau_\ve, H_\ve)$, so that $F(X^i_\ve) = X$.  
Note that in this notation, $X$ is fixed for all $H_\ve \in \H(h)$.

When we view $H_\ve$ as a hole for the open system, 
if $k < \tau(X^i_\ve)$ is the first time that $f^k(X^i_\ve) \subset H$, we define 
$\mathring \tau_\ve(X^i_\ve) = k$; otherwise if $X^i_\ve$ returns to $X$ before
encountering $H$, we set $\mathring \tau_\ve(X^i_\ve) = \tau_\ve(X^i_\ve)$.
The induced map for the open system $\hF_\ve$ is defined similarly
so that $\hF_\ve(x) = F_\ve(x)$ whenever $x \in X$ returns to $X$ without entering $H$
along the way.  

Now for  $H \in \H(h)$, we have a tower $(f_\Delta, \Delta(H))$ constructed so that
$\pi^{-1}H$ is a union of 1-cylinders $\tH_{\ell,j}$.  We recall 
that the family of towers corresponding
to $\H(h)$ satisfy (P1)-(P4) with uniform constants.  
 For our function space $\B$ on the tower,
we choose $\beta < \min \{\alpha, t_0 \log \xi \}$ as in Section~\ref{ssec:uniform p1-p4} 
and add the further requirement
that $\beta < \alpha - \varsigma s_t$, where $\varsigma$ is from \eqref{eq:slow approach} and 
$s_t$ is the scaling exponent from Lemma~\ref{lem:scale}.  Note that with this choice of 
$\beta$, the exponent $\varsigma$ necessarily satisfies
\begin{equation}
\label{eq:varsigma}
\varsigma < \min \{ (\alpha - \beta)/s_t, 2 \vartheta_c \} .
\end{equation}
Also, if we need to shrink $\beta$ in what follows, this does not affect the value
of $\varsigma$, which is fixed and depends on $z$.
 
Let $g^H_t \in \B$ denote the 
eigenfunction corresponding to $\lambda^H_t$ for $\Lp_{\phi^H_t} := \hLp_{\phi_{\Delta,t}}$.  
We have introduced this new notation for the transfer operator with the hole 
in order to make dependence on $H$ explicit.  We drop the subscript $\Delta$ since
all the objects we work with in this section will be on the tower.  We denote by
$g^0_t \in \B$ the invariant probability density for the transfer operator without the
hole, $\Lp_{\phi_t}$.  We remark that $g^0_t$ also depends on $H$ since different
$H$ induce different towers $\Delta(H)$, but the projection $\pa_\pi g^0_t$ is independent of
$H$.

We use the notation $\hDelta^n(H)$ to indicate the set of points that has not escaped $\Delta(H)$
by time $n$ and set $\hDelta^0(H) = \hDelta(H)$. 
Recall the notation $\tH = \pi^{-1}H$,  $\tH = \cup_\ell \tH_\ell$ and
$\tH_\ell = \cup_j \tH_{\ell,j}$.

With our definitions, $g^H_t \equiv 0$ on $\tH$ and all columns above each $\tH_{\ell,j}$.
Since it will be convenient to make our estimates directly on $\tH$, we
extend $g^H_t$ to $\tH$ by $g^H_t(x) = (\lambda^H_t)^{-1} g(f_\Delta^{-1}x)$ for
$x \in \tH$.  Note that this extended version of $g^H_t$ is still 0 in the column above 
each $\tH_{\ell,j}$.  We normalize $g^H_t$ so that,
$\int_{\Delta(H)} g^H_t \, d\bm_t =1$.

If we redefine $\Lp_{\phi^H_t} \psi := \Lp_{\phi_t} (1_{\hDelta(H)} \psi)$ (rather than
$\Lp_{\phi_t} (1_{\hDelta^1(H)} \psi)$), it extends the operator
so that $\Lp_{\phi^H_t} g^H_t(x) = \lambda^H_t g^H_t(x)$ for
$x \in \tH$.  Note, however, that $\Lp_{\phi^H_t}$ still does not let mass map out of $\tH$ so that
no mass maps to the columns above $\tH$.  We will use this extended definition of
$\Lp_{\phi^H_t}$ for the remainder of this section.

Since $\e(m_t, H_\ve) = -\log \lambda^{H_\ve}_t$,  
proving Theorem~\ref{thm:zero hole} is equivalent to estimating
\[
\lim_{\ve \to 0} \frac{1-\lambda^{H_\ve}_t}{\mu_t(H_\ve)} ,
\]
which we now start to do.

We begin with the key observation that by definition of 
$\Lp_{\phi_t^H}$ and $g^H_t$,
\[
\lambda^H_t = \lambda^H_t \int_{\Delta(H)} g_t^H \, d\bm_t = \int_{\Delta(H)} \Lp_{\phi_t^H} g_t^H \, d\bm_t
= \int_{\hDelta^0(H)} g_t^H \, d\bm_t.
\]
So now for any $n \ge 0$, using the conditional invariance of $g^H_t$, we write
\begin{equation}
\label{eq:hole}
\begin{split}
1 - \lambda^H_t & = \int_{\Delta(H)} g_t^H \, d\bm_t - \int_{\hDelta^0(H)} g_t^H \, d\bm_t 
= \int_{\tH} g_t^H \, d\bm_t \\
& = (\lambda^H_t)^{-n} \int_{\tH} (\Lp_{\phi_t^H}^n g_t^H - \Lp_{\phi_t}^n g_t^0) \, d\bm_t 
+ (\lambda^H_t)^{-n} \int_{\tH} g_t^0 \, d\bm_t  .
\end{split}
\end{equation}

The following lemmas, the proofs of which we give later, will allow us to prove 
Theorem~\ref{thm:zero hole}. 
 
\begin{lemma}
\label{lem:spectral gap}
The transfer operators $\Lp_{\phi_t}$ and $\Lp_{\phi^H_t}$ have a uniform spectral gap
for all $H \in \H(h)$ with $h$ sufficiently small.  More precisely, there exist $C_2 >0$ and
$\sigma_0 < 1$ such that for all $\psi \in \B$ and $n \ge 0$,
\[
\begin{split}
  \| \Lp^n_{\phi_t} \psi - c_0(\psi) g^0_t \|_\B & \le C_2 \sigma_0^n \| \psi \|_\B \| g^0_t \|_\B
  \; \; \; \mbox{and} \\
  \| (\lambda^H_t)^{-n} \Lp_{\phi^H_t}^n \psi - c_H(\psi) g^H_t \|_\B 
 &  \le C_2 \sigma_0^n \| \psi \|_\B \| g^H_t \|_\B,
  \end{split}
\]
where the constants $c_H(\psi)$ and $c_0(\psi)$ represent the magnitude of the
projections of $\psi$ onto the eigenspaces spanned by $g^H_t$ and $g^0_t$, respectively.
Moreover, $c_H(\psi) \to c_0(\psi)$ as $\diam(H) \to 0$ in $\H(h)$. 
\end{lemma}

\begin{lemma}
\label{lem:uniform hole}
There exists $C_3 >0$ such that  if  $z$ satisfies \eqref{eq:slow approach} then for  $h$ sufficiently small and for each $H \in \H(h)$, 
\[
\sum_{\ell \ge 0} e^{\beta \ell} \bm_t(\tH_\ell) \le
C_3 \mu_t(H) .
\]
\end{lemma}

\begin{lemma}
\label{lem:typical}
Condition \eqref{eq:slow approach} is generic with respect to both $\mu_t$ and $m_t$.
\end{lemma}

The three lemmas above are sufficient for the proof of the generic case.  
The next lemma is required in the periodic case.

\begin{lemma}
\label{lem:per}
Suppose that $z$ is a periodic point of period $p$ and $\frac{d\mu_t}{dm_t}(z)\in (0,\infty)$.
Given $n\in \N$ we can choose $h$ so small that if $H\in \H(h)$ then 
\[
\mu_t(H)e^{S_p\phi_t(z)}(1-\gamma(n)) 
\le \int_{\tH} (\Lp_{\phi_t}^n-\Lp_{\phi^H_t}^n) g_0~d\bm_t 
\le \mu_t(H)e^{S_p\phi_t(z)}(1+\gamma(n)) 
\]
where $\gamma(n)\to 0$ as $n\to \infty$ independently of $H \in \H(h)$.
\end{lemma}

\begin{proof}[Proof of Theorem~\ref{thm:zero hole}]
{\em Non-periodic case.}
Fix $\ve >0$ and choose $h>0$ so small that $\lambda^H_t \ge \sigma_0^{1/2}$
for all $H \in \H(h)$.  Next choose $n$ so large that $\sigma_0^{n/2} < \ve$.  Finally,
choose $H$ sufficiently small that $(\lambda^H_t)^{-n} \le 1+ \ve$ and
$f_\Delta^{-k}\tH \cap \tH = \emptyset$ for all $1 \le k \le n$.  This last choice
is possible using the aperiodicity of $z$.

Note that due to this last choice, we have $\Lp^n_{\phi^H_t} = \Lp^n_{\phi_t}$
when integrated over $\tH$.  Now the first term on the right hand side of
\eqref{eq:hole} becomes
\[
\begin{split}
(\lambda^H_t)^{-n} & \int_{\tH} \Lp_{\phi_t}^n (g_t^H - g_t^0) \, d\bm_t
\le (\lambda^H_t)^{-n} \sum_{\ell \ge 1} e^{\beta \ell} \| \Lp_{\phi_t}^{n} g_t^H - g_t^0 \|_\B \, \bm_t(\tH_\ell) \\
& \le C_2 \sigma_0^{n/2} \| g^H_t \|_\B \| g^0_t \|_\B \sum_{\ell \ge 1} e^{\beta \ell} \, \bm_t(\tH_\ell) 
\le C_2 \ve \| g_t^H \|_\B \| g_t^0 \|_\B \, \mu_t(H),
\end{split}
\]
where we have used Lemmas~\ref{lem:spectral gap}-\ref{lem:typical}
as well as the fact that  $g^0_t$ is uniformly bounded above and below on $\Delta$
(see \cite[Prop.~2.4]{BDM}). 
Note that  in the application of Lemma~\ref{lem:typical}, 
we use the fact that $c_0(g^H_t) = \int_{\Delta(H)} g^H_t \, dm_t = 1$.
Now $\| g_t^H \|_\B$ and  $\| g_t^0 \|_\B$ are uniformly bounded for all
$H \in \H(h)$ with constants depending only on (P1)-(P4) \cite[Prop.~2.3]{BDM}.
Thus the first term of \eqref{eq:hole} can be made an arbitrarily small multiple of $\mu_t(H)$.

The second term of \eqref{eq:hole} is simply bounded by
\begin{equation}
\label{eq:easy}
\int_{\tH} g^0_t \, d\bm_t \le (\lambda^H_t)^{-n} \int_{\tH}  g_t^0 \, d\bm_t
\le (1+\ve) \int_{\tH} g^0_t \, d\bm_t .
\end{equation}
Since $\int_{\tH} g^0_t \, d\bm_t = \bmu_t(\tH) = \mu_t(H)$, we have shown that
\[
\frac{1-\lambda^H_t}{\mu_t (H)} = 1 + \epsilon(H),
\]
where $\epsilon(H) \to 0$ as $m_t(H) \to 0$.  

{\em Periodic case.}
We split the first term on the right side of \eqref{eq:hole} into two and renormalize $g^0_t$
by $c_H(g^0_t)$:
\begin{equation}
\label{eq:periodic split}
\begin{split}
1 - \lambda^H_t 
& = (\lambda^H_t)^{-n} c_H(g^0_t)^{-1} \int_{\tH} \Lp_{\phi_t^H}^n ( c_H(g^0_t) g_t^H - g_t^0) \, d\bm_t \\
& \; \; \; + (\lambda^H_t)^{-n} c_H(g^0_t)^{-1} \int_{\tH} (\Lp_{\phi_t^H}^n - \Lp_{\phi_t}^n) g_t^0 \, d\bm_t 
   + (\lambda^H_t)^{-n} c_H(g^0_t)^{-1} \int_{\tH} g_t^0 \, d\bm_t  .
\end{split}
\end{equation}

As before, we fix $\ve >0$ and choose $n$ so large that $\sigma_0^{n/2} < \ve$ and
$\gamma(n)$ from Lemma~\ref{lem:per} is less than $\ve$.
Next choose $h>0$ so small that $\lambda^H_t \ge \sigma_0^{1/2}$ and 
$(\lambda^H_t)^{-n} \le 1+ \ve$
for all $H \in \H(h)$.  Finally, since $c_0(g^0_t) =1$, by Lemma~\ref{lem:spectral gap},
we may shrink $h$ further so that $c_H(g^0_t) \in (1-\ve, 1+\ve)$
and for our given $n$, the holes are 
small enough that the conclusion of Lemma~\ref{lem:per} holds.

Using Lemmas~\ref{lem:spectral gap} and \ref{lem:uniform hole} again,  the first term in
\eqref{eq:periodic split} is bounded by
\[
\left| (\lambda^H_t)^{-n} c_H(g^0_t)^{-1} \int_{\tH} \Lp_{\phi_t^H}^n (c_H(g^0_t) g_t^H - g_t^0) \, d\bm_t \right|
 \le C \ve \| g_t^H \|_\B \| g_t^0 \|_\B \, \mu_t(H) .
\]
Using Lemma~\ref{lem:per}, the second term in \eqref{eq:periodic split} is bounded above and below by
\[
 - \mu_t(H)e^{S_p \phi_t(z)} \frac{(1+ \ve)^2}{1-\ve}
\le \frac{(\lambda^H_t)^{-n}}{c_H(g^0_t)} \int_{\tH} (\Lp_{\phi_t^H}^n - \Lp_{\phi_t}^n) g_t^0 \, d\bm_t 
\le - \mu_t(H) e^{S_p \phi_t(z)} \frac{1-\ve}{1+\ve} .
\]
Finally, the third term of \eqref{eq:periodic split} is bounded above and below as in 
\eqref{eq:easy}.

Putting these three estimates together, we conclude,
\[
1- e^{S_p \phi_t(z)} - C \ve \le \frac{1 - \lambda^H_t }{\mu_t(H)} \le 1- e^{S_p \phi_t(z)} + C \ve ,
\]
for a constant $C$ independent of $H$.   Since $\ve$ is arbitrary, this completes the
proof of the theorem.
\end{proof}

The proof of Theorem~\ref{thm:zero hole adapted} requires some minor adaptations of our preparatory lemmas, so we will leave its proof until we have proved those lemmas.


\subsection{Proofs of generic lemmas}
\label{lemma proofs}

\begin{proof}[Proof of Lemma~\ref{lem:spectral gap}]
The proof follows from the fact that for the transfer operator without the hole
$\Lp_{\phi_t}$, $e^{-\beta}$ gives an upper bound on the second largest eigenvalue as well
as a bound on the essential spectral radius. This is proved in \cite[Theorem 1.4 and Section 4.1]{maume}.  There it is
shown that in our setup (since $\beta$ is very close to 0), a constructive bound on 
the second largest eigenvalue is given by $\tanh (\mathcal{R}/2)$ where
$\mathcal{R} = \log \frac{1+e^{-\beta}}{1-e^{-\beta}}$.  Simplifying this expression yields
$\tanh (\mathcal{R}/2) = e^{-\beta}$.

By our Proposition~\ref{prop:uniform tails}, we have uniform control of the tails of the return time $\tau$ 
for all $H \in \mathcal{H}(h)$ and $t \in [t_0,t_1]$.  Thus we work with a fixed $\beta>0$
(chosen in Section~\ref{prep}) 
in all our towers which 
gives a uniform bound on the second largest eigenvalue of $\Lp_{\phi_t}$.

The fact that the spectrum and spectral projectors
of $\Lp_{\phi^H_t}$ are close to that of $\Lp_{\phi_t}$ outside the
disk of radius $e^{-\beta}$ follows from \cite[Lemma 3.6]{demers wright}
(which in turn is an application of \cite{keller liverani} adapted to sequences
of Young tower constructions).  There,
it is shown that the eigenvalues of $\Lp_{\phi^H_t}$ and $\Lp_{\phi_t}$ outside the
disk of radius $e^{-\beta}$ vary by at most $\mathcal{O}(h^\epsilon)$ 
for all $H \in \H(h)$ and some $\epsilon >0$.

Choosing $h$ sufficiently small, we may guarantee that $\lambda^H_t \ge e^{-\beta/3}$
and the second largest eigenvalue of $\Lp_{\phi^H_t}$ is at most $e^{-2\beta/3}$.
This ensures that $\sigma_0 <1$ in the statement of the lemma is at most
$e^{-\beta/3}$ for all $H \in \H(h)$.

In addition, letting $\Pi_{\lambda^H}$ and $\Pi_1$ denote the
spectral projections onto the eigenspaces associated with $g^H_t$ and $g^0_t$, respectively,
the same perturbative results from \cite{demers wright}  used above imply that
for $\psi \in \B$,
\[
\begin{split}
|c_H(\psi) - c_0(\psi)| & = \left| \int_{\Delta(H)} \Pi_{\lambda^H} \psi - \Pi_1 \psi \, dm_t \right| \\
& \le \|  \Pi_{\lambda^H} \psi - \Pi_1 \psi \|_\B \sum_\ell m_t(\Delta_\ell) e^{\beta \ell}
\le C h^\epsilon,
\end{split}
\]
where $C$ is independent of $H \in \H(h)$, proving the continuity of $c_H(\psi)$
 in $H$.
\end{proof}

Before giving the proof of Lemma~\ref{lem:uniform hole}, we need the following important fact
about the scaling of the measure $m_t$ on small sets.  

\begin{lemma}
\label{lem:scale}
Suppose that $z$ is either a $\mu_t$-typical point or a periodic point with $\frac{d\mu_t}{dm_t}(z)\in (0,\infty)$ and set $s_t:=t+\frac{p_t}{\lambda(z)}$.  Then for each $\eps>0$, there exists $\delta(\eps)>0$ such that  for all  $\delta\in (0, \delta(\eps))$, 
\begin{equation}
(2\delta)^{s_t+\eps}\le m_t(B_\delta(z))\le (2\delta)^{s_t-\eps}.
\label{eq:scale}
\end{equation}
Moreover, $s_t>0$, and in the case that $z$ is $\mu_t$-typical then $s_t$ can also be written as $\frac{h(\mu_t)}{\lambda(\mu_t)}$, which is $<1$ whenever $t\neq 1$.
\end{lemma}

\begin{proof}
In the case that $z$ is $\mu_t$-typical, \eqref{eq:scale} follows immediately from the definition of local dimension $d_{\mu_t}$ at $x$, where
$$d_{\mu_t}(x)=\lim_{\delta\to 0} \frac{\log\mu_t(B_\delta(x))}{\log 2\delta},$$
whenever the limit exists.  For $\mu_t$-typical points this is always equal to the dimension of the measure  $\frac{h(\mu_t)}{\lambda(\mu_t)}$, see \cite{Hofdim}.  So $s_t=\frac{h(\mu_t)}{\lambda(\mu_t)}=t+\frac{p_t}{\lambda(\mu_t)}=t+\frac{p_t}{\lambda(z)}$.  By the Ruelle-Pesin formula, as well as the obvious fact that the dimension of a measure is bounded by the dimension of the space, $s_t\le 1$.  Moreover, since the unique invariant measure which has dimension 1 is $\mu_1$, we have $s_t<1$ whenever $t\neq 1$, see for example \cite{Led81}.  We can make the switch from the invariant measure $\mu_t$ to the conformal measure $m_t$ using the fact that the density at typical $z$ exists and takes a value in $(0, \infty)$.

In the case when $f^q(z)=z$, first notice that elementary arguments on the pressure function imply that  for any $t$ where an equilibrium state of positive entropy $\mu_t$ exists, $p_t>-t\lambda(z)$ for any periodic point $z$.  Therefore, $s_t$ in this case is strictly positive.  Moreover, $S_q\phi(z)<0$.

For the scaling properties of $m_t$ around $z$, the situation is in many ways simpler than the typical case, although we can't call on the powerful theory of local dimension described above.  The proof is similar to \cite[Lemma 4.1]{FreFreTod13} so we only sketch it.   We use the fact that we can pick $\hat\delta>0$ such that on $B_{\hat\delta}(z)$, $|Df^q|\sim  |Df^q(z)|$,
where $\sim$ denotes a uniform constant depending only on $\hat\delta$ and 
the distortion constant
$C_d$ from (A1).  
Then for a ball of size $\delta\in (0, \hat\delta)$, conformality implies that for $n=\lfloor \frac{\log(\hat\delta/\delta)}{\log|Df^q(z)|}\rfloor$,
$$m_t(B_\delta(z))\sim |Df^{nq}(z)|^{-t}e^{-nqp_t} m_t(B_{\hat\delta}(z)) \sim \delta^te^{-nqp_t}.$$

Hence 
$$\lim_{\delta\to 0} \frac{\log m_t(B_\delta(x))}{\log 2\delta}=t+\frac{qp_t}{\log|Df^q(z)|}=t+\frac{p_t}{\lambda(z)},$$
as required.
\end{proof}

\begin{proof}[Proof of Lemma~\ref{lem:uniform hole}]
Holes in $\Delta$ are created in one of two ways:  When
$f^nX$ encounters $H$ during a bound period or when it is free.  We split
the relevant sum into these pieces,
\[
\sum_\ell \bm_t(H_\ell) e^{\beta \ell} = \sum_{\mbox{\scriptsize bound}} \bm_t(H_\ell) e^{\beta \ell}
+ \sum_{\mbox{\scriptsize free}} \bm_t(H_\ell) e^{\beta \ell} .
\]

\noindent
{\em Estimate on bound pieces.}  
Since $\bm_t(H_\ell) = \bm_t(f_\Delta^{-\ell}H_\ell)$, we will estimate the sum
over all $1$-cylinders $X^i_\ve$ such that $f^\ell(X^i_\ve) \subset H$ and $X^i_\ve$ is bound
at time $\ell$, $\ell < \tau(X^i_\ve)$.

If $X^i_\ve$ is bound at
time $\ell$, then $|f^\ell(x) - f^\ell(c)| \le \delta_0 e^{-2\vartheta_c \ell}$, for each $x \in X^i_\ve$.
Thus fixing $x \in X^i_\ve$ and using \eqref{eq:slow approach}, we obtain
\[
\delta_z e^{- \varsigma \ell} \le |f^n(c) - z| \le |f^n(c) - f^n(x)| + |f^n(x) - z|
\le \delta_0 e^{- 2\vartheta_c \ell} + |H|/2 ,
\]
Since $\varsigma < 2 \vartheta_c$ by \eqref{eq:varsigma}, 
this inequality can only be satisfied by sufficiently large $\ell$
and if $\delta_z < \delta_0$, by finitely many small values of $\ell$ as well.
So we assume the worst case scenario, that $\delta_z< \delta_0$.
The finitely many $\ell$ must satisfy 
$\ell \le \frac{\log(\delta_0/\delta_z)}{2\vartheta_c - \varsigma}$.
On the other hand, the sufficiently large $\ell$ must satisfy,
\[
\delta_z e^{-\varsigma \ell} \le |H|/2 \implies \ell \ge \frac{- \log(|H|/2\delta_z)}{\varsigma} .
\]

Putting these estimates together, we have the following estimate on the contribution
from bound pieces,
\begin{equation}
\label{eq:bound est}
\begin{split}
\sum_{\mbox{\scriptsize bound}} \bm_t(H_\ell) e^{\beta \ell}
& \le \sum_{\ell \le  \frac{\log(\delta_0/\delta_z)}{2\vartheta_c - \varsigma}}
\bm_t(H_\ell) e^{\beta \ell}
+ \sum_{\ell \ge  \frac{- \log(|H|/2\delta_z)}{\varsigma}}
\bm_t(H_\ell) e^{\beta \ell} \\
& \le C e^{\beta \frac{\log(\delta_0/\delta_z)}{2\vartheta_c - \varsigma}} \sum_\ell \bmu_t(H_\ell)
+ \sum_{\ell \ge  \frac{- \log(|H|/2\delta_z)}{\varsigma}} Ce^{-(\alpha-\beta) \ell} \\
& \le C e^{\beta \frac{\log(\delta_0/\delta_z)}{2\vartheta_c - \varsigma}} \mu_t(H)
+ C \delta_z^{(\beta-\alpha)/\varsigma} |H|^{(\alpha-\beta)/\varsigma} ,
\end{split}
\end{equation}
where we have used the fact that $\bmu_t$ has density with respect to $\bm_t$
uniformly bounded above and below on $\Delta$ in the last line.  By \eqref{eq:varsigma}, 
the exponent of $|H|$ in the last term is greater than $s_t = t + \frac{p_t}{\lambda(z)}$
from Lemma~\ref{lem:scale}.  So remembering that $H_\ve = B_\ve(z)$
and by choosing $h$ sufficiently small, we have 
$|H|^{(\alpha-\beta)/\varsigma} \le m_t(H) \le C\mu_t(H)$, where we have used that
the density of $\mu_t$ is bounded away from 0.
This completes the estimate on bound pieces.

\noindent
{\em Estimate on free pieces.}  For free pieces, we adapt the estimates in
\cite[Section 6]{DHL} and their modification due to the extra cutting by $\partial H$ in
\cite[Lemma 4.5]{BDM}.  We fix $n$ and estimate the mass of one-cylinders
$\omega = X_i$ which are free when they enter $H$ for the first time at time $n$. 

According to the construction in \cite{DHL},
each one-cylinder $\omega$ is contained in a sequence of nested intervals 
$\omega \subset \omega^{(j)} \subset \omega^{(j-1)} \subset \cdots \subset \omega^{(1)}$
and corresponding times $s_1, \ldots, s_j$ such that $|f^{s_i}(\omega^{(i)})| \ge \delta_1$.
Thus there is $\delta_t >0$ such that $m_t(f^{s_i}(\omega^{(i)})) \ge \delta_t$
for each $i = 1, \ldots, j$.  We call the times $s_i$ {\em growth times} for $\omega$.

Define $E^n_j$ to be the set of one-cylinders $\omega$ such that $f^n\omega \subset H$
for the first time, $f^n\omega$ is free,
and $\omega$ belongs to an interval which grows to fixed length $\delta_1$ precisely $j$
times before time $n$.  Then 
\[
\sum_{\substack{f^n\omega \subset H \\ \mbox{\scriptsize free}}} m_t(\omega) e^{\beta n}  
= \sum_{\substack{\omega \in E^n_j \\ j \le \zeta n}} m_t(\omega) e^{\beta n} + \sum_{\substack{\omega \in E^n_j \\ \zeta n < j \le n}} m_t(\omega) e^{\beta n},
\]
where $0 < \zeta < 1$ is determined below.

For $j \le \zeta n$, we follow the proof of \cite[Lemma 10]{DHL} and \cite[Lemma 4.5]{BDM}
to define $\{ s_1 = r_1 \}$ as the set of points for which the first growth to 
length $\delta_1$ occurs at time
$r_1$.  We have by Proposition~\ref{prop:uniform tails}, 
\[
m_t(s_1 = r_1) \le C_0e^{-\alpha r_1} \le \frac{C_0}{\delta_t} e^{-\alpha r_1} m_t(X) .
\]
We then repeat this estimate on each element $f^{r_1}\omega$, which has $m_t$-measure at least 
$\delta_t$ by definition of $r_1$.  Thus
\[
m_t(x \in f^{r_1}\omega) : s_1 \ge r_2 ) \le C_0 e^{-\alpha r_2} \le \frac{C_0}{\delta_t} e^{-\alpha r_2}
m_t(f^{r_1}\omega) .
\]
By bounded distortion, this comprises a comparable fraction of the set in $\omega$, and thus
\[
m_t(s_2 = r_1 + r_2 : s_1 = r_1 ) \le \frac{C_0^2 D_\delta}{\delta_t^2} e^{-\alpha(r_1+r_2)} ,
\]
where $D_\delta$ is the distortion constant.
Iterating this $j$ times, we have
\begin{equation}
\label{eq:growth times}
m_t(s_j = r_1 + \cdots r_j : s_1 = r_1, \ldots, s_{j-1} - s_{j-2} = r_{j-1}) \le \frac{C_0^j D_\delta^{j-1}}{\delta_t^j} e^{-\alpha s_j} m_t(X).
\end{equation}

Next we focus on the intervals $f^{s_j}\omega$ which lie inside $f^{s_j}\omega^{(j)}$ for a fixed 
$\omega^{(j)}$.   In particular, we need to control the increase in complexity between
times $s_j$ and $n$, \ie the number of subintervals of $f^{s_j}\omega^{(j)}$ that will
overlap when they enter $H$ at time $n-s_j$.

In $n-s_j$ iterates, $f^{s_j}\omega$ will enter $H$ for the first time.
Along the way, due to the definition of $s_j$, 
$f^i(f^{s_j}\omega^{(j)})$ cannot grow to length greater than $\delta_1$ or 
have a piece that makes a return to $X$, for $i=1, \ldots, n-s_j$.  
If $f^i(f^{s_j}\omega^{(j)}) \subset B_\delta(c)$ for some $c \in \Crit_c$, then a doubling may occur
creating an overlap of subintervals in $f^{s_j}\omega^{(j)}$ when they enter $H$.  We need to show that the expansion gained from time $s_j$ to time $n$ is sufficient to
overcome this growth in complexity.

Let $p_\delta$ denote the minimum length of a bound period for $x \in B_\delta(c)$.
According to \cite[Lemma 2]{DHL}, for $x \in B_\delta(c)$, we have $|Df^{p+1}(x)| \ge \kappa^{-1} e^{\theta(p+1)}$ when $x$ reaches the end of its bound period of length $p$, where
\[
\theta = \frac{\Lambda - 5 \vartheta_c \ell_c}{2 \ell_c} - \frac{\Lambda - 5 \vartheta_c \ell_c}{2 \ell_c p_\delta} > 0.
\]
Note that the second term can be made arbitrarily small by choosing $\delta$ to be small (and
therefore $p_\delta$ large).  We choose $p_\delta$ sufficiently large that 
$e^{\theta p_\delta} > 2^{1/t}$.  Define $\bar\theta = \theta - \frac{\log 2}{t p_\delta}$.

Now suppose that $\omega \subset \omega^{(j)}$ makes $k$ visits to $B_\delta(c)$ between times
$s_j$ and $n$ and is free at time $n$.   
Then concatenating the expansion from \cite[Lemma 2]{DHL}
and (C1), we have for $x \in f^{s_j}\omega$,
\[
|Df^{n-s_j}(x)| \ge \kappa \delta^{\ell_{\max} -1} e^{\bar \gamma (n-s_j)} 
2^{(1+\frac{1}{p_\delta})\frac{k}{t}},
\]
where $\bar\gamma = \min \{ \gamma, \bar\theta \}$.

Since the complexity of $f^i(f^{s_j}\omega^{(j)})$ increases at most by a factor of 2 with
each entry into $B_\delta(c)$, we fix $\omega^{(j)}$ and let $A_k$ denote those 
$\omega \in E^n_j$, $\omega \subset \omega^{(j)}$, that make $k$ visits to $B_\delta(c)$
between times $s_j$ and $n$.  Then
\begin{equation}
\label{eq:after s_j}
\begin{split}
\sum_{\substack{\omega \subset \omega^{(j)} \\ \omega \in E^n_j}} m_t(f^{s_j} \omega) 
& \le \sum_{k=0}^{n-s_j} 
\sum_{\omega \in A_k} C 2^{-(1+ \frac{1}{p_\delta})k} e^{-\bar\gamma (n-s_j)t} e^{p_t(n-s_j)}
m_t(f^n\omega) \\
& \le \sum_{k=0}^{n-s_j} C 2^k 2^{-(1+ \frac{1}{p_\delta})k} e^{-\bar\gamma (n-s_j)t} e^{p_t(n-s_j)}
m_t(H) \\
& \le Ce^{(n-s_j)(-\bar\gamma t + p_t)} m_t(H),
\end{split}
\end{equation}
for some $C >0$. 

Then since $m_t(f^{s_j}\omega^{(j)}) \ge \delta_t$, we iterate use \eqref{eq:after s_j} to
iterate \eqref{eq:growth times} one more time to obtain,
\[
m_t(s_1 = r_1, s_2 -s_1 = r_2, \ldots, n-s_j = r_{j+1}) \le \frac{C_0^{j+1} D_\delta^{j}}{\delta_t^{j+1}} e^{-\theta_1 n} m_t(H),
\]
where $\theta_1 = \min \{ \bar \gamma t - p_t, \alpha \}$.
Summing over all possible $(j+1)$-tuples such that $\sum_i r_i =n$, we use the same
combinatorial argument as in \cite[Lemma 7]{DHL} to bound their number by $e^{\eta n}$,
where $\eta$ can be made as small as we like by choosing $\zeta$ 
sufficiently small (but holding $\delta$ fixed, which allows us to hold the distortion constant
fixed).  Thus,
\begin{equation}
\label{eq:less than zeta}
\sum_{\substack{\omega \in E^n_j \\ j \le \zeta n}} m_t(\omega) e^{\beta n}
\le \frac{C_1}{\delta_t} \left( \frac{C_1 D_\delta}{\delta_t} \right)^{\zeta n} e^{(-\theta_1 + \beta + \eta)n}
m_t(H) ,
\end{equation}
and choosing $\zeta$ and $\beta$ sufficiently small yields a bound exponentially small in $n$ times
$m_t(H)$.

Finally, we focus on those $\omega$ with $j > \zeta n$.  Here we follow the proof
of \cite[Lemma 11]{DHL} and its modification in \cite[Lemma 4.5]{BDM}.  By \cite[Lemma 1]{DHL},
every time a piece grows to length $\delta_1$, a fixed fraction, call it $\xi$, of 
$f^{s_i}(\omega^{(i)})$ makes a full return to 
$X$ by a fixed time $s^*$. 
Since $m_t(X) \ge \delta_t$, we also know the portion that 
makes a full return by time $s^*$ constitutes a fixed fraction
$\xi_t$ of the $m_t$ measure of $f^{s_i}(\omega^{(i)})$.
Due to bounded distortion, a fixed fraction $\xi_t/D_\delta$
of $\omega^{(i)}$ makes a return
by time $s_i+s^*$.  We now iterate this $j$ times, using the fact that 
each $\omega \in E^n_j$ belongs to an interval $\omega^{(j)}$
which also has its $j$th growth time at time $s_j$.  Thus  
\[
\sum_{\omega \in E^n_j} m_t(\omega) \le \sum_{\omega^{(j)} \in E_j} m_t(\omega^{(j)})
 \le \left(1- \frac{\xi_t}{D_\delta} \right)^{j} m_t(X) ,
\]
where $E_j$ is the set of $\omega^{(j)}$ corresponding to $E^n_j$.
  Note that once $\mathcal{H}(h)$ is fixed, neither $X$ nor
$D_\delta$ changes as we shrink $h$.  Moreover, the fraction $\xi_t$ that returns
to $X$ by time $s^*$ does not deteriorate as $h$ decreases since a smaller hole does not prevent
an interval from making its full return to $X$.

Now due to bounded distortion and letting $|Df^{s_j}(\omega^{(j)})|$ denote the average value of
$|Df^{s_j}|$ on $\omega^{(j)}$, we have $m_t(\omega^{(j)}) \ge D_\delta^t m_t(f^{s_j}\omega^{(j)})
|Df^{s_j}(\omega^{(j)})|^{-t} e^{-s_jp_t}$.  This, together with the previous estimate implies
\begin{equation}
\label{eq:deriv sum}
\sum_{\omega^{(j)} \in E_j} |Df^{s_j}(\omega^{(j)})|^{-t} e^{-s_j p_t} \le D_\delta^{-t} 
\left( 1 - \frac{\xi_t}{\D_\delta} \right)^j .
\end{equation}

Between time $s_j$ and time $n$, the complexity of $\omega$ entering $H$ can increase
in the same way as described earlier.  Thus we may combine \eqref{eq:after s_j} 
with \eqref{eq:deriv sum} to obtain
\[
\begin{split}
\sum_{\omega \in E^n_j} m_t(\omega)
& \le \sum_{\omega^{(j)} \in E_j}  
\sum_{\substack{\omega \subset \omega^{(j)} \\ \omega \in E^n_j}}
\frac{m_t(\omega)}{m_t(f^{s_j}\omega)} m_t(f^{s_j}\omega) \\
& \le \sum_{\omega^{(j)} \in E_j}  
D_\delta^{-t} |Df^{s_j}(\omega^{(j)})|^{-t} e^{-s_j p_t} 
\sum_{\substack{\omega \subset \omega^{(j)} \\ \omega \in E^n_j}}
m_t(f^{s_j}\omega) \\
& \le C D_\delta^{-2t} 
\left( 1 - \frac{\xi_t}{\D_\delta} \right)^{j} e^{(n-s_j)(-\bar\gamma t + p_t)} m_t(H) .
\end{split}
\]
Summing this estimate for $j > \zeta n$ yields
\[
\sum_{\substack{\omega \in E^n_j \\ \zeta n < j \le n}} m_t(\omega) e^{\beta n}
\le C' \left(1- \frac{\xi_t}{D_\delta} \right)^{\zeta n } e^{\beta n} m_t(H),
\]
and this can be made exponentially small in $n$ by choosing $\beta$ sufficiently small.
Note that choosing $\beta$ small will force $H$ to be very small, but this is not a 
restriction since we are interested only in the small hole limit.

This estimate combined with \eqref{eq:less than zeta} completes the estimate on the free pieces
and the proof of Lemma~\ref{lem:uniform hole}.
\end{proof}

\begin{proof}[Proof of Lemma~\ref{lem:typical}]
Fix $0 < \varsigma < s_t$ and $\eta >0$.  By Lemma~\ref{lem:scale}, there exists $\delta_\varsigma >0$ and
a measurable set $E_{\delta_\varsigma}$ with $m_t(E_{\delta_\varsigma}) > 1-\eta$
such that \eqref{eq:scale} holds
for all $z \in E_{\delta_\varsigma}$ and all $\delta < \delta_\varsigma$. 

Now choose $\delta$ so small that  $2\delta < \delta_\varsigma$.  Then if
$z \in E_{\delta_\varsigma} \cap B_{\delta e^{-\varsigma n}}(f^n(c))$, we have by Lemma~\ref{lem:scale},
\[
m_t(B_{\delta e^{-\varsigma n}}(f^n(c))) \le m_t(B_{2\delta e^{-\varsigma n}}(z))
\le (4\delta)^{s_t - \varsigma} e^{-n \varsigma (s_t - \varsigma)} .  
\]
Let $J = \{ n \in \N : B_{\delta e^{-\varsigma n}}(f^n(c)) \cap E_{\delta_\varsigma} \neq \emptyset \}$.
Then it follows from the above estimate that
\[
m_t(\cup_{n \in J}  B_{\delta e^{-\varsigma n}}(f^n(c))) 
\le \sum_{n \in J}  (4\delta)^{s_t - \varsigma} e^{-n \varsigma (s_t - \varsigma)}
\le (4\delta)^{s_t - \varsigma} \frac{1}{1- e^{- \varsigma (s_t - \varsigma)}} ,
\]
and by shrinking $\delta$, we may make the quantity on the right hand side less
than $\eta$.  Now we estimate for all such $\delta$ sufficiently small,
\[
\begin{split}
m_t(z \in [0,1] : & z \notin B_{\delta e^{-\varsigma n}}(f^n(c)) \; \forall n \in \N )
\ge m_t(z \in E_{\delta_\varsigma} : z \notin B_{\delta e^{-\varsigma n}}(f^n(c)) \; \forall n \in \N ) \\
& \ge m_t(z \in E_{\delta_\varsigma}) - m_t(z \in E_{\delta_\ve} : z \in B_{\delta e^{-\varsigma n}}(f^n(c)) 
\mbox{ for some } n \in \N ) \\
& \ge m_t(z \in E_{\delta_\varsigma}) - m_t(\cup_{n \in J}  B_{\delta e^{-\varsigma n}}(f^n(c))) 
\ge 1-2\eta .
\end{split}
\]
Since $\eta>0$ was arbitrary, this completes the proof of the lemma with respect to
$m_t$.  Since $\mu_t \ll m_t$, the property is generic with respect to $\mu_t$ as well.
\end{proof}


\subsection{Proof of periodic lemma}
\label{density}

In this section we prove the necessary estimate to conclude Theorem~\ref{thm:zero hole}
in the periodic case.

\begin{proof}[Proof of Lemma~\ref{lem:per}]
The main idea of this proof is that by selecting $H$ appropriately, 
$\int_{\tH} (\Lp_{\phi_t}^n-\Lp_{\phi^H_t}^n) g^0_t~d\bm_t$ is comparable to the measure 
of the set $\hat H:=H\cap f^{-p}(H)$.  Once we have shown this, 
we use the fact that the density at $z$ exists and lies in $(0,\infty)$ to deduce that 
\[
\frac{\mu_t(\hat H)}{\mu_t(H)}\sim \frac{m_t(\hat H)}{m_t(H)}\sim e^{S_p\phi(z)},
\]
where the final estimate is immediate by conformality, 
and $\sim$ denotes a uniform constant
depending only on $f$ and $t$ (not $H$).

Recall that by construction, if a domain $X_\eps^i$ has an iterate $k$ such that 
$f^k(X_\eps^i)\cap H\neq \es$ then $f^k(X_\eps^i)\subset H$.  Hence if we fix 
a column $i$, then every time that an iterate $f_\Delta^k(\Delta_{i, 0})$ projects to intersect $H$, then in fact $\pi (f_\Delta^k(\Delta_{i, 0}))\subset H$.
We set $\mu_{\Delta,t} = g^0_t \bm_t$.

Note that $\Lp^n_{\phi^H_t}$ only includes preimages of points in $\tH$
which enter $\tH$ for the first time at time $n$, while $\Lp^n_{\phi_t}$ counts all preimages
of points in $\tH$ which enter $H$ at time $n$.
Thus to estimate the quantity $\int_{\tH} (\Lp_{\phi_t}^n-\Lp_{\phi^H_t}^n) g^0_t~d\bm_t$,
we sum the $\mu_{\Delta,t}$ measure on the tower of the set of points which 
both project to the hole $H$ at time $n$, as well as doing so at some previous time 
$0\le k \le n-1$.  

Fixing $n$, we choose $h$ so small that for $H \in \H(h)$, if  $x\in H$, but $f^p(x)\notin H$ then 
$f^k(x)\notin H$ for $k=1, \ldots, n$.  In particular, this means that if a point in a column 
of $\Delta$ projects to $H$ then the only way an $f_\Delta$-iterate of $x$ can 
project to $H$ again before time $n$ is if $\pi x$ was actually in some subset 
$f^{-ip}(H)\cap H$ for $i\ge 1$.  Moreover, the $f_\Delta$-orbit of $x$ cannot 
return to the base $\Delta_0$ and then later project to the hole again before time $n$
since $H$ is so small that it cannot grow to length $\delta_1$ by time $n$.  

Fix a column $i$ and suppose that at some level the projection 
$f_\Delta^k(\Delta_{i, 0})$ is inside $H$.  If this only happens once, 
then we don't count it.  If it happens exactly twice then we count the 
set of points which when iterated forwards $n$ times project to the hole for the second time.  
Notice that since $\mu_{\Delta, t}$ is $f_\Delta$-invariant, the set of points which 
when iterated forwards $n$ times projected to the hole for the first time has the 
same measure as $\Delta_{i, 0}$, i.e., $\mu_{\Delta, t}(\Delta_{i, 0})$.  
Continuing in the same way, we see that if parts of the column project to the 
hole exactly $k$ times, then we measure $(k-1)\mu_{\Delta, t}(\Delta_{i, 0})$.  
Note that we only continue up to time $n$, so since within a column we can 
only repeatedly hit the hole every $p$ iterates, this process stops when 
$k>\frac np$.  Therefore, 
\begin{equation}
\int_{\tH} (\Lp_{\phi_t}^n-\Lp_{\phi^H_t}^n) g^0_t~d\bm_t=\sum_{k= 2}^{\lfloor\frac np\rfloor} (k-1)\sum_{\{i:\pi \Delta_{i, j_q}\subset H \text{ for } k \text{ times } j_q\}}\mu_{\Delta, t}(\Delta_{i, 0}).
\label{eq:trans defect}
\end{equation}

Now notice that our setup implies that, as for $H$ itself, if $\pi \Delta_{i, j}\cap \hat H\neq \es$ then $\pi \Delta_{i, j}\subset \hat H$, so the measure of $\hat H$ is by definition
$$\sum_{\{i,j:\pi \Delta_{i, j}\subset \hat H\}} \mu_{\Delta, t}(\Delta_{i, j}).$$
If we consider this sum column by column, we can separate it into terms where the column projects to $\hat H$ $k$ times, and noting that each element $\Delta_{i, j}$ has the same $\mu_{\Delta, t}$-measure, we obtain (suggestively using index $k-1$ rather than $k$) 
\begin{equation}
\mu_t(\hat H) = \sum_{k\ge 2} (k-1)\sum_{\{i:\pi \Delta_{i, j_q}\subset \hat H \text{ for } k-1 \text{ times } j_q\}}\mu_{\Delta, t}(\Delta_{i, 0}).
\label{eq:meas inn hole}
\end{equation}
But since once a domain has $\pi \Delta_{i, j_q}\subset \hat H$ and $\pi f_\Delta^p( \Delta_{i, j_q})\cap \hat H=\es$, we must also have $\pi f_\Delta^p( \Delta_{i, j_q})\subset H\sm \hat H$, and moreover, $f_\Delta^p( \Delta_{i, j_q})$ is still in column $i$.  
Therefore, the values in \eqref{eq:trans defect} and \eqref{eq:meas inn hole} are the same, 
up to the measure of $H\cap f^{-p \lfloor\frac np\rfloor}(H)$ which we claim 
is of order $e^{S_{n}\phi(z)}$  and  is exponentially small in $n$.

To see this, note that by (H2) and choice of $\delta_0$, the orbit of $z$ is always
`free.'  Thus by (C1), we have 
$|Df^i(z)| \ge \kappa \delta_0^{\ell_{\max}-1} e^{\gamma i}$ for any $i \in \N$.
A similar bound holds at time $n$ for any $x \in H \cap f^{-p\lfloor \frac np \rfloor}(H)$
using bounded distortion, and the claim follows. 
\end{proof}


\subsection{Specific classes of maps satisfying our assumptions: the proof of Theorem~\ref{thm:zero hole adapted}}

We next assert that there is a reasonable class of maps with periodic points satisfying 
\eqref{eq:slow approach}.  Note that we expect the conclusions of the following 
lemma to hold for a much larger class of maps and periodic points.

\begin{lemma}
Let $z_4$ be a repelling periodic point of $f_4$ not lying on the critical orbit.  Consider its hyperbolic continuation $z_\lambda$ for $\lambda$ close to 4 (i.e. $z_\lambda$ has a topologically identical orbit under $f_\lambda$ as  $z_4$ does under $f_4$).  Then for any $\varsigma>0$ there exist $t_0<1<t_1$ and a positive Lebesgue measure set of parameters $\Omega'=\Omega'(z_4)$ such that whenever $\lambda\in \Omega'$ and $t\in [t_0, t_1]$, then $f_\lambda$ has an equilibrium state $\mu_t$ and 
there exists 
$ \delta_{z_\lambda} >0$ such that 
$$ |f^n(c) - z_\lambda| \ge \delta_{z_\lambda} e^{-n\varsigma} \mbox{ for all } n \ge 0.$$
In particular,  \eqref{eq:slow approach} holds for $z_\lambda$.
\label{lem:Omega per}
\end{lemma}

\begin{proof}
The proof is the same as for \cite[Theorem 7]{FreFreTod13}.  There it is shown that there is an acip for $f_\lambda$, but since this has exponential tails, \cite{BTeqnat} implies that $\mu_t$ also exists.  Moreover, 
they show that \eqref{eq:orig Now cond} holds. 
\end{proof}

\begin{proof}[Proof of Theorem~\ref{thm:zero hole adapted}]
We focus on the periodic case, since the tools required for the generic case are almost classical.  This means that we wish to prove condition (P) for our family of maps $f_\lambda$ and periodic points $z_\lambda$, as well as noting that $\hat\delta$ in Lemma~\ref{lem:scale} can be taken uniformly.  We consider the family of maps $\Omega'(z_4)$ given by Lemma~\ref{lem:Omega per}.  We notice that if we fix the constants $\nu_c, \Lambda, \alpha, \beta,  t$ then, possibly by restricting our class of maps, we also get a uniform estimate on $s_t$ and so by Lemma~\ref{lem:Omega per},  Lemma~\ref{lem:uniform hole} holds throughout our family.
We also use the fact here that $\mu_t$ and $m_t$ do not change too much within this family due to statistical stability (see \cite{FreTod10}), so the constants coming from the measure $m_t$ of small intervals can also be taken to depend only on the family.  Hence \eqref{eq:slow approach} holds.

To complete the proof of the lemma we must show that the density $\frac{d\mu_{\lambda, t}}{dm_{\lambda, t}}$ is bounded at $z_\lambda$.  This follows almost exactly as in \cite{Now93}, in particular Corollary 4.2.  The problem was expressed there as finding a uniform bound on  $(\L_{\phi_1}^n1)(z)$. 
The main issue was to estimate the distortion of $f^n$ along orbits which are relevant to this transfer operator, which was guaranteed when
\begin{equation}\sum_{n=0}^\infty\frac1{|Df^n(f(c))|^{\ell_c} |f^n(f(c))-z|^{1-\frac{1}{\ell_c}}}<\infty.
\label{eq:orig Now cond}
\end{equation}
Clearly this holds in our case by \eqref{eq:slow approach} and the exponential growth of derivative along the critical orbit.
For our case,  
for each $t$ in a neighborhood of 1 we are interested in finding a uniform bound on 
$(\L_{\phi_t}^n1)(z)$, that is showing
\begin{equation*}
\sum_{n=0}^\infty\frac1{\left(|Df^n(f(c))|^{\ell_c} |f^n(f(c))-z|^{1-\frac{1}{\ell_c}}\right)^t e^{np(t)}}<\infty.
\label{eq:Now cond}
\end{equation*}
Clearly for $t$ close to 1, the fact that this is bounded holds analogously to the case when $t=1$, i.e., \eqref{eq:orig Now cond} above.
\end{proof}

\begin{remark}
Note that the above proof of the boundedness of the density was closely tied to \eqref{eq:slow approach}.  The proof in  \cite{Now93} requires a negative Schwarzian condition along with unimodality.  We would expect this to extend beyond that setting.
\end{remark}

\end{document}